\documentclass[english,reqno]{amsart}

\usepackage{microtype}
\usepackage{babel}
\usepackage[T1]{fontenc}
\usepackage[utf8]{inputenc}
\usepackage{lmodern}
\usepackage{csquotes}
\usepackage{mathtools}
    \numberwithin{equation}{section}
    \mathtoolsset{showonlyrefs}
\usepackage{amsthm}
\usepackage{amssymb}
\usepackage[mathscr]{euscript}
\usepackage[hidelinks]{hyperref}
\usepackage[giveninits=false,
    url=false,
    style=alphabetic,
    backref=true,
    isbn=false]{biblatex}
\usepackage{zref-clever}
    \zcsetup{abbrev=true}
\usepackage{tikz}
\tikzset{
    cell/.style={
        anchor=south west,
        draw,
        minimum size=1cm,
    },
}
\usepackage[all]{xy}

\theoremstyle{plain}
\newtheorem{lemma}{Lemma}[section]
\newtheorem{theorem}[lemma]{Theorem}
\newtheorem{conjecture}[lemma]{Conjecture}
\newtheorem{corollary}[lemma]{Corollary}
\newtheorem{proposition}[lemma]{Proposition}

\theoremstyle{remark}
\newtheorem{remark}[lemma]{Remark}

\theoremstyle{definition}
\newtheorem{definition}[lemma]{Definition}

\AddToHook{env/theorem/begin}{%
    \zcsetup{countertype={lemma=theorem}}}
\AddToHook{env/conjecture/begin}{%
    \zcsetup{countertype={lemma=conjecture}}}
\AddToHook{env/corollary/begin}{%
    \zcsetup{countertype={lemma=corollary}}}
\AddToHook{env/proposition/begin}{%
    \zcsetup{countertype={lemma=proposition}}}
\AddToHook{env/remark/begin}{%
    \zcsetup{countertype={lemma=remark}}}
\AddToHook{env/definition/begin}{%
    \zcsetup{countertype={lemma=definition}}}
\zcRefTypeSetup{conjecture}{
Name-sg = Conjecture ,
name-sg = conjecture ,
Name-pl = Conjectures ,
name-pl = conjectures ,
}

\newcommand\coh{\mathrm{Coh}}
\newcommand\LL{\mathsf{L}} 
\newcommand\tl{\tilde{\lambda}}
\newcommand\ep{\epsilon}
\newcommand\al{\alpha}
\newcommand\q{T} 

\newcommand\tn{{\tilde{n}}}
\newcommand\tb{\tilde{b}}
\newcommand\tK{\tilde{K}}
\newcommand\anew{\tilde{b}'}

\newcommand\cO{\mathcal{O}}
\newcommand\cE{\mathcal{E}}

\newcommand\calN{{\mathcal N}}

\newcommand{\calE}{{\mathcal E}}
\newcommand{\calF}{{\mathcal F}}
\newcommand{\calO}{{\mathcal O}}

\newcommand{\calH}{{\mathcal H}}
\newcommand{\calL}{{\mathcal L}}
\newcommand{\calI}{{\mathcal I}}

\newcommand{\calQ}{{\mathcal Q}}
\newcommand{\calC}{{\mathcal C}}
\newcommand{\calG}{{\mathcal G}}
\newcommand{\calK}{{\mathcal K}}
\newcommand{\calB}{{\mathcal B}}
\newcommand{\calT}{{\mathcal T}}
\newcommand{\mred}{\mathcal P} 
\newcommand\mcou{\mathcal M} 

\newcommand\BC{\mathbb{C}}
\newcommand\BZ{\mathbb{Z}}
\newcommand\BP{\mathbb{P}}
\newcommand\BR{\mathbb{R}}
\newcommand{\IQ}{{\mathbb Q}}
\newcommand{\IC}{{\mathbb C}}
\newcommand{\IZ}{{\mathbb Z}}
\newcommand{\IP}{{\mathbb P}}
\newcommand{\IA}{{\mathbb A}}

\newcommand{\scrF}{{\mathscr F}}

\newcommand{\PT}{\mathrm{PT}}

\newcommand\eu{\mathrm{e}} 
\newcommand{\grp}[1]{\mathrm{#1}} 
\newcommand\br[1]{[{#1}]}

\newcommand\qq{\mathfrak{q}} 

\newcommand\QQ{\mathsf{Q}} 

\DeclareMathOperator\ch{ch}
\DeclareMathOperator\td{td}
\DeclareMathOperator\rk{rk}
\DeclareMathOperator\tr{tr}
\DeclareMathOperator\vol{vol}
\DeclareMathOperator\End{End}
\DeclareMathOperator\Hom{Hom}

\DeclareMathOperator\cok{cok}

\newcommand{\be}{\begin{equation}}
\newcommand{\ee}{\end{equation}}
\newcommand\T{\rule{0pt}{2.6ex}}

\title[Quasimaps of surfaces]
{Coulomb branch localization, quasimaps,
and surface counting in Calabi--Yau fourfolds}
\author{D.-E. Diaconescu \and N. Piazzalunga}

\begin{document}

\begin{abstract}
\noindent
We present a string theoretic approach to surface counting in local
Calabi--Yau fourfolds via supersymmetric localization in topologically
twisted four-dimensional gauge theories. This approach is based on a
spectral correspondence between ${\rm PT}_1$-stable pairs on local
fourfolds and twisted quasimaps with fixed two-dimensional domain
associated to the ADHM quiver, or, equivalently,  ADHM sheaves.
For local toric fourfolds, we derive a conjectural residue formula for
the $K$-theoretic quasimap partition function via Coulomb branch localization.
As a result, in this case, we obtain a conjectural prescription fixing all
usual sign ambiguities in the equivariant computation of such invariants.
We present some explicit computations for local $\BP^2$, extending
the results available in the literature, and describe the formalism in
general. This is the first instance of Coulomb branch localization for
a quasimap theory in the context of four-dimensional gauge theories.

\end{abstract}

\maketitle
\tableofcontents

\section{Introduction}

The study of Donaldson--Uhlenbeck--Yau (DUY) equations on four-manifolds
and their higher-dimensional generalizations has been a fruitful avenue
of research both in geometry and physics.  Physically, they control
the non-perturbative dynamics of supersymmetric gauge theories, while
mathematically the moduli spaces of solutions can be used to define
enumerative invariants in various dimensions.  However, it is the
interface of the two disciplines where many striking results have been
achieved, often within the framework of string theory.

We will be concerned here with a generalization of the DUY equations
occurring naturally in the Donaldson--Thomas (DT) theory of Calabi--Yau
fourfolds \cite{Corrigan:1982th, Donaldson:1996kp}. More precisely we will
study the Pandharipande--Thomas (PT) stable pair theory of Calabi--Yau
fourfolds \cite{St_pairs_fourfolds,Counting_surfaces_I}.  This amounts to
counting stable pairs $(\calF,s)$ consisting of a compactly supported
purely two-dimensional sheaf $\calF$ and a generically surjective
section. This is called $\PT_1$-theory \cite{Counting_surfaces_I}
since the cokernel of $s$ is required to be at most one-dimensional.
From a physical viewpoint, we are counting D8-D4-D2-D0 configurations
\cite{Nekrasov:2017cih, Nekrasov:2018xsb} via
supersymmetric indices in the associated eight-dimensional supersymmetric
gauge theory.  Although this is not the focus of the present paper,
it should be noted that the relation between DT and $\PT_1$ theory is
not presently understood in detail for Calabi--Yau fourfolds.

The present paper showcases a novel Coulomb branch localization approach
to the $\PT_1$ theory of local Calabi--Yau fourfolds fibered over compact
complex surfaces. In a nutshell, the relation between fourfold stable
pair invariants and Coulomb branch partition functions is encoded in
\zcref{fig:flow-chart}.  The main outcome of the displayed logical tree
is a conjectural expression for the $K$-theoretic $\PT_1$ partition
function of local toric Calabi--Yau fourfolds in terms of meromorphic
integrals on the Coulomb branch of an ADHM gauge theory. We will refer
to this relation as the $\PT_1$/Coulomb correspondence.  We emphasize
that the Coulomb branch partition function is not obtained by direct
equivariant localization computations on moduli spaces; it provides
instead an alternative residue formula derived through a gauge theory
construction. An important feature of this correspondence is that it
also provides an intrinsic prescription fixing the usual sign ambiguities
present in localization computations for Calabi--Yau fourfolds.  Below we
will explain the main steps in \zcref{fig:flow-chart} in more detail.

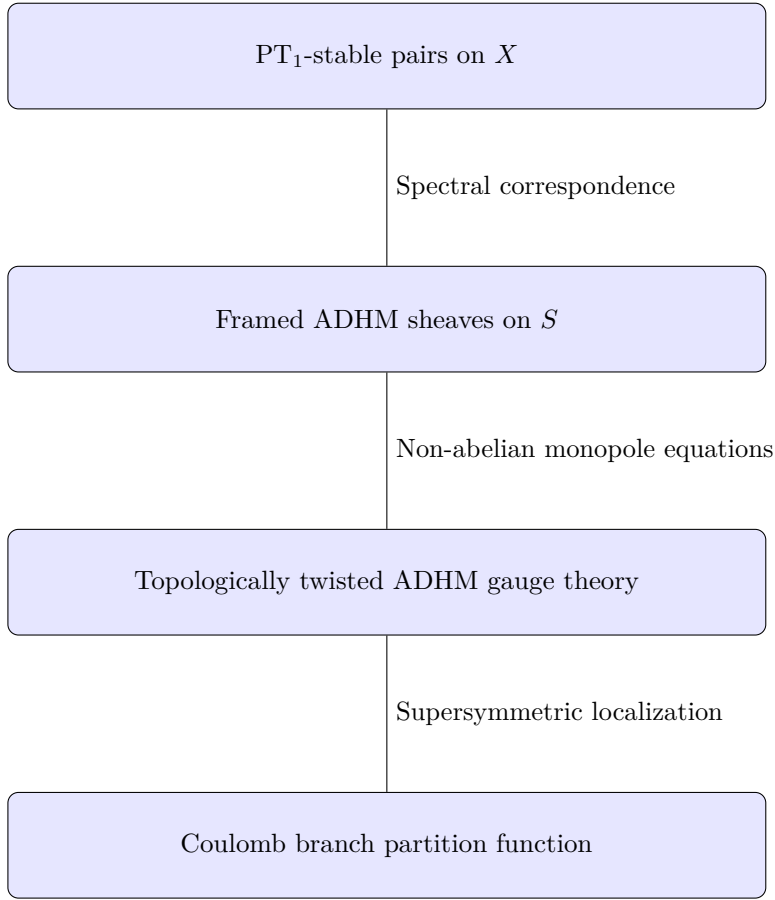
\begin{figure}[htb]
\centering
\tikzstyle{block} = [rectangle, draw, fill=blue!10,
 text width=28em, text centered, rounded corners, minimum height=4em]
\tikzstyle{line} = [draw]
\begin{tikzpicture}[node distance = 3.5cm, auto]
    \node [block] (init) {$\PT_1$-stable pairs on $X$};
    \node [block, below of=init] (identify) {Framed ADHM sheaves on $S$};
    \node [block, below of=identify] (evaluate) {Topologically twisted ADHM gauge theory};
    \node [block, below of=evaluate] (decide) {Coulomb branch partition function};
    \path [line] (init) -- node[right] {Spectral correspondence}node [left] {}(identify);
    \path [line] (identify) -- node {Non-abelian monopole equations } (evaluate);
    \path [line] (evaluate) --node {Supersymmetric localization } (decide) ;
\end{tikzpicture}
\caption{ $\PT_1$/Coulomb correspondence.}
\label{fig:flow-chart}
\end{figure}

\subsection{The setup}\label{sec:setup}
For us, a local Calabi--Yau fourfold $X$ will be the total space of a
direct sum of holomorphic line bundles $\LL_1 \oplus \LL_2$ over a smooth
projective surface $S$ so that $\LL_1 \otimes \LL_2 \simeq \omega_S$.
Here $\omega_S = \calO_S(K_S)$ denotes the dualizing sheaf of $S$.

It will be often convenient to regard compactly supported coherent
sheaves on $X$ as coherent sheaves on its natural projective bundle
completion, $Y = \BP (\calO_S \oplus \LL_1 \oplus \LL_2)$.  Note that
coherent sheaves will be sometimes referred to simply as sheaves, since
all sheaves considered in this paper will be coherent.

Unless otherwise stated, we will assume the vanishing condition
$H^1(S,\IQ)=0$.  As shown in \zcref{sec:topinv}, if this is the case,
the Chern classes of any nonzero two-dimensional sheaf $\calF$ on $Y$
with compact support contained in $X$ are given by 
\[
\ch_2(\calF)= r \sigma_*[S] \qquad 
\ch_3(\calF) = \sigma_*(\beta) 
\]
for a unique  pair $(r,\beta) \in \IZ \times H_2(S;Q)$ with $r \geq
1$. Here $[S]\in H_4(S)$ is the fundamental cycle, and $\sigma: S \to Y$
is the unique section of $Y$ over $S$ that factors through the zero
section of $X$. Moreover Chern classes of sheaves on $Y$ are regarded
as homology classes.  We will use the notation
\[
v(\calF) \coloneq (r, \beta, n)\in \IZ_{\geq 1} \oplus H_2(S) \oplus \IZ.
\]

For fixed topological invariants, one can construct
\cite{St_pairs_fourfolds, Counting_surfaces_I} a fine quasi-projective
moduli space $SP_1(X;v)$ parametrizing $\PT_1$-stable pairs $(\calF,
s)$ on $X$ with compact support.
Moreover, as shown in loc.~cit.\ this moduli space admits an isotropic
perfect obstruction theory \cite{OT-isotropic}. In particular, the
moduli space admits a virtual cycle, as well as a virtual structure
sheaf ${\calO}_{SP_1(X,v)}^{vir}$. We will denote by
\[ 
SP_1(X) = \sqcup_{v} SP_1(X;v)
\]
the disjoint union over all topological invariants.  We also note
that the moduli space $SP_1(X)$ admits a natural derived enhancement
$\mathbf{SP}_1(X)$ induced by the derived structure on the moduli space
of objects in $D^b(Y)$.

In order to address compactness issues, we will work equivariantly with
respect to a torus action $T \times X \to X$ acting linearly on the
fibers of the projection $\pi:X\to S$. This yields a $T$-action on the
moduli space $SP_1(X; v)$ with compact fixed locus. Therefore one can
define virtual $K$-theoretic invariants
\be \label{eq:PTinvA}
\PT_1^K(X;v) = \chi(\widehat \calO_{SP_1(X,v)}^{vir}
 \otimes {\hat \Lambda}_M^\bullet(Rp_*\scrF))
\ee
by equivariant localization, where $\scrF$ is the universal sheaf over
$SP_1(X)\times S$ and $p: SP_1(X)\times S \to SP_1(X)$ the canonical
projection.  The hat symbol means the usual twist by the square root of the
determinant bundle, converting the index of Dolbeault operator to Dirac.
For fixed $r\geq 1$, we will denote by
\be\label{eq:PTpartfctA} 
Z_r^{\PT_1} (X;\QQ,\qq)=
\sum_{(\beta, n)\in H_2(S) \times \IZ}
\PT_1^K(X;r,\beta,n) \QQ^\beta \qq^n 
\ee
the associated partition function.

Note that the fiberwise anti-diagonal action $\IC^\times \times X \to
X$ provides an example of a Calabi--Yau torus action for any surface
$S$. In this case, the fixed locus coincides with the zero section.
In specific examples, one can construct higher-rank Calabi--Yau torus
actions on $X$ provided that the surface $S$ admits a nontrivial torus
action $T_S \times S \to S$ so that the line bundles $\LL_1$, $\LL_2$
have equivariant structure.  In particular, this is the case if
$S$ is a toric surface as in \zcref{sec:4d}.

The main goal of this paper is to gain insight in the structure of
the partition function \eqref{eq:PTpartfctA} exploiting its physical
interpretation as a counting function for supersymmetric D8-D4-D2-D0
bound states.  On general grounds, we expect the above index to be
encoded in the low-energy effective action of D4-branes supported on
the zero section, which is a topologically twisted ADHM gauge theory.
From a mathematical perspective, the connection between stable pair
theory and gauge theory is made rigorous through ADHM sheaves on $S$,
by analogy with the local threefold case \cite{ADHMsheaves}.

Note that a similar approach to local curves in Calabi--Yau fourfolds
was developed \cite{Qmaps_potential} using a quasimap theory associated
to quivers with potential, which corresponds to certain 3d gauge
theories \cite{One_leg_vertex}.  Furthermore, the relation between Coulomb
branch partition functions and enumerative quasimap invariants was
studied \cite{3d_qmaps} in the framework of three-dimensional gauge
theories. As explained in more detail in \zcref{sec:PTCcorresp}, the
main novelty in four-dimensional gauge theories is that the Coulomb
branch partition function includes instanton effects.

\subsection{ADHM sheaves and spectral correspondence}

In this paper, a (rank one) framed ADHM sheaf on $S$ will be defined by
the data $(\calE, \Phi_1, \Phi_2, \phi)$ where
\begin{itemize}
\item $\calE$ is a torsion-free coherent sheaf on $S$, and
\item $\Phi_i: \calE \to \calE \otimes \LL_i$, $1\leq i\leq 2$, and $\phi:
\calO_S\to \calE$ are morphisms of coherent sheaves on $S$ satisfying
the relation
\[
(\Phi_1 \otimes \mathbf{1}_{\LL_2}) \circ \Phi_2
- (\Phi_2\otimes \mathbf{1}_{\LL_1}) \circ \Phi_1 =0.
\]
\end{itemize}
In addition, we will impose a stability condition that requires the data
$(\calE, \Phi_1, \Phi_2, \phi)$ to be generically cyclic.  A framed ADHM
sheaf will be called locally free if $\calE$ is locally free.

\begin{remark}
Note that the generic stability condition exhibits ADHM sheaves as
generalization of twisted quasimaps with fixed two-dimensional domain.
In comparison to one-dimensional domain, a distinguishing feature of the
present construction is the presence of points on $S$ where the underlying
torsion-free sheaf $\calE$ is not locally free.  This is expected to lead
to new geometric phenomena, especially for quasimaps with moving domain,
which will be discussed elsewhere.
\end{remark}

Using a relative variant of Beilinson's theorem (\zcref{thm:tilting})
\zcref{lem:stpairsX} provides a one-to-one correspondence between
$\PT_1$ stable pairs on $S$ with compact support and framed ADHM
sheaf on $S$. This is a generalization of the well-known spectral
correspondence for Higgs sheaves on curves, as well as the variant proven
\cite{TT_VW} for local surfaces in threefolds.  Then the main result
of \zcref{sect:SP_ADHM} (cf.\ \zcref{thm:triplesADHM,cor:pairsADHM}) reads
\begin{theorem} \label{thm:mainA}
The derived algebraic space parametrizing $\PT_1$ stable pairs on
$X$ is naturally identified with a derived moduli space of framed ADHM
sheaves on $S$ via spectral correspondence.
\end{theorem}

\subsection{PT1/Coulomb correspondence} \label{sec:PTCcorresp} 

As a consequence of \zcref{thm:mainA}, the virtual counting problem for
$\PT_1$ stable pairs on $X$ admits an equivalent formulation in terms
of framed ADHM sheaves on $S$. More precisely, the partition function
\zcref{eq:PTpartfctA} is identified with an analogous partition function
$Z_r^\mathrm{ADHM}(\qq,\QQ)$ for $K$-theoretic virtual equivariant ADHM
sheaf invariants.

By the DUY principle, locally free ADHM sheaves are in one-to-one
correspondence to solutions of certain nonabelian monopole equations
derived from twisted ADHM gauge theory on $S$ \cite[section
8]{Nekrasov:2023nai}. Therefore the ADHM sheaf partition function
provides a rigorous construction for the Higgs branch localization of the
partition function of the topologically twisted ADHM gauge theory on $S$.
By the general principles of supersymmetric localization, we expect this
partition function to admit a Coulomb branch presentation in terms of
meromorphic integrals.

The novelty in the present context is that four-dimensional
Coulomb branch localization must include instanton effects,
as opposed to its lower-dimensional counterparts.  For pure
$\calN=2$ gauge theories on toric surfaces, this problem was
first addressed by \textcite{Nakajima:2003pg} in a special case,
by \textcite{Nekrasov:2006otu} for general non-compact manifolds,
and by \textcite{Bonelli:2020xps} in some compact cases, leading to a
novel Coulomb branch approach to Donaldson four-manifold invariants.
In the present paper we extend Coulomb branch localization to ADHM gauge
theories on toric surfaces and verify agreement with existing explicit
results for $\PT_1$ invariants \cite{Counting_surfaces_II}.

In the following $S$ will be a smooth compact toric surface, hence $X$
will also be a toric Calabi--Yau fourfold. As explained in detail in
\zcref{sec:toricdata} one has a natural rank-four torus action $T_X
\times X\to X$ with finite fixed locus, preserving the image of the zero
section $S\to X$. By definition, the Calabi--Yau torus $T\subset T_X$ is
stabilizer of the Calabi--Yau structure. Since the $T$-action preserves
the zero section, it induces a torus action on $S$ with constant rank-one
stabilizers.  Moreover, the line bundles $\LL_1$, $\LL_2$ have a
natural $T$-equivariant structure.

For the physically minded reader, a summary of the gauge theory
computation is presented below by comparison with the case of non-compact
manifolds.  Since we are concerned with $K$-theoretic invariants, we
consider the natural 5d lift of the 4d gauge theory.

Using the notation in \zcref{sec:4d}, by cluster decomposition, one first
computes the instanton partition function $z_{inst} (\tb)$ in a fixed
vacuum state ${\tb}$ as a sum over topological sectors $\ch_2(F)$ of
equivariant integrals (twisted Witten indices)
\begin{equation}
\tr_{\calH} (-1)^F \eu^{-\beta H} U_g =
\int_{\mcou_r} \hat A \ch \wedge^\bullet_M \calE
\end{equation}
for a fixed value of the equivariant fluxes. As usual, for non-compact
surfaces,  $\mcou_r$ is a moduli space of torsion-free sheaves with
framing along the divisor at infinity.  The instanton partition function
depends on $\tb$ (and other equivariant parameters) via $g$, an element
of the maximal torus, out of which one builds a unitary operator $U_g$
acting on the Hilbert space $\calH$.

We then remove the framing by integrating over $\tb$.  In other words,
a new field becomes dynamical if $S$ is compact, and we integrate over it.
Its contribution is the perturbative part $z_{pert}$, which is finite for
a compact manifold.   We find it useful to make a change of variables
in the integrand, so that it takes the same form for different values
of the equivariant fluxes in a given topological sector, and we get the
meromorphic contour integral
\begin{equation}
Z^C_r \coloneq \sum_\xi \oint d\anew \,
z_{cl} (\xi) \cdot z_{inst} (\anew \eu^{-\xi}) \cdot z_{pert}~,
\end{equation}
where the contour depends on which theory we consider.  Since the
construction is quite involved, we refer to \zcref{sec:4d} for more
details.  Here, it suffices to summarize some of its main features.

\begin{itemize}
\item As explained in \zcref{sec:coulombfixed},
the torus action $T \times X \to X$ translates into a torus action 
on the fields in gauge theory, hence the associated equivariant 
parameters are naturally present in the expression of $Z_r^C$. 
\item By construction, the Coulomb branch partition function admits a natural
formal series expansion in terms of two gauge theory counting parameters:
$\qq$, associated to the instanton number, and $\QQ$, associated to magnetic flux.
\item As explained at the beginning of \zcref{sec:4d},
we assume that $\LL_1$ is such that $\chi(S,\LL_1)$ has no constant terms,
where $\chi(S, \LL_1)$ is the $T$-equivariant Euler characteristic.
\end{itemize}

Then the main claim of  $\PT_1$/Coulomb correspondence is the
\begin{conjecture} \label{conj:main}
For any $r \geq 1$, the Coulomb branch partition function $Z^C_r (\QQ,
\qq)$ coincides with the generating function $Z_r^{\PT_1} (\QQ,\qq)$
of $T$-equivariant $\PT_1$-invariants, up to an overall numerical factor
depending only on $r$.
\end{conjecture}


\begin{remark}
As explained in \zcref{rem:zeromodes}, the partition function
admits a well-defined non-equivariant limit provided that
the line bundle $\LL_1$ satisfies the vanishing conditions
\be \label{eq:vancond}
H^k (\LL_1) = 0, \ k=0,2.
\ee
By Serre duality, this is equivalent to $H^0(\LL_i)=0$, $1\leq i\leq 2$,
in which case the moduli space $SP_1(X;v)$ is compact by \zcref{prop:compC}.
Hence \zcref{eq:PTpartfctA} admits a non-equivariant
limit as well, and \zcref{conj:main} implies that their
non-equivariant limits coincide, up to an overall numerical factor
depending only on $r$.
\end{remark}

In \zcref{sec:localp2}, we verify \zcref{conj:main} by explicit computations
(first in a few equivariant cases, and then more extensively in the non-equivariant limit)
for the case where $S$ is the projective plane $\BP^2$ and
$\LL_i \simeq \calO_{\BP^2}(-i)$ for $1\leq i\leq 2$.  Note that, up to
permutation, this is the only choice of line bundles that satisfies the
condition $\chi(\LL_1)=0$.  Furthermore, for $S= \BP^2$, this is also the
only case where the moduli space of stable pairs $SP_1(X;v)$ is compact
for any topological invariants $v$.  As shown in \zcref{sec:localp2},
upon taking the non-equivariant limit, we obtain the following results
for the Coulomb branch partition function:
\begin{multline}\label{eq:p2rk1}
Z^C_1 = \QQ^{\frac{3}{2}} \qq \br{M^{-1}}
+ \QQ^{5/2} \left( \qq^3 \br{M^{-3}} + \qq^2 \br{M^{-2}} + 2 \qq \br{M^{-1}} \right) \\
+ \QQ^{7/2} \left( \qq^6 \br{M^{-6}} + \qq^5 \br{M^{-5}} + 2 \qq^4 \br{M^{-4}} +3 \qq^3 \br{M^{-3}} + 5 \qq^2 \br{M^{-2}} + \dots \right) \\
+ \QQ^{9/2} \qq^{10} \br{M^{-10}} + \dots
\end{multline}
where dots mean higher orders, as well as
\begin{multline} \label{eq:p2rk2}
- Z^C_2 \cdot \br{M}^{-4} = \QQ^4 \qq^4
+ \QQ^5  \Big( 2 \qq^7 (M^{\frac12}+M^{\frac12}) (2M + 1 + 2M^{-1}) \\
+ 2 \qq^6 (M + M^{-1})
+ 2 \qq^5 (M^{\frac12} + M^{-\frac12})
+ 4 \qq^4 \Big) \\
+ \QQ^6 \Big( 2 \qq^{11} (M^{\frac12}+M^{-\frac12}) (4M^3+3M^2+6M+4+6M^{-1}+3M^{-2}+4M^{-3}) \\
+ \qq^{10} (7M^3+12M^2+15M+16+15M^{-1}+12M^{-2}+7M^{-3}) \\
+ 2 \qq^9 (5M+1+5M^{-1})(M+M^{-1})(M^{\frac12}+M^{-\frac12}) \\
+ \qq^8 (11M^2+20M+28+20M^{-1}+11M^{-2}) \\
+ 2 \qq^7 (11M + 3 + 11M^{-1}) (M^{\frac12} + M^{-\frac12})
+ \dots \Big) + \dots
\end{multline}
In order to compare the above expressions with the partition function
\eqref{eq:PTpartfctA}, note that the class $\beta \in H_2(\BP^2, \IQ)$
is written as $\beta = m [H]$ for some $m \in \frac12 \IQ$, where $[H]
\in H_2(\BP^2, \IZ)$ is the hyperplane class.  Then the above expressions
are in agreement with $\PT_1$-stable pair invariants for all cases
previously computed \cite{Counting_surfaces_II}, providing at the same
time new conjectural results for higher-degree invariants.

\subsection{Open questions}
We conclude this section with a brief discussion of some interesting
open problems.

Geometrically, this work provides a first encounter with a quasimap
theory defined on surfaces rather than curves. Motivated by the
enumerative geometry of surfaces in Calabi--Yau fourfolds, this
construction occurs naturally in topologically twisted ADHM gauge theory.
As a natural generalization, we expect similar constructions to be
present in more general topologically twisted quiver gauge theories as
compactifications of moduli spaces of solutions to non-abelian monopole
equations. Furthermore, we also expect supersymmetric localization to
provide an explicit conjectural formula for the associated partition
functions via meromorphic integrals.  For affine ADE quivers, we also
expect the resulting Coulomb branch partition function to be related to
the $\PT_1$-theory of local Calabi--Yau orbifolds.

Another natural open problem is whether one can prove \zcref{conj:main}
using the wallcrossing approach \cite{K_Donaldson},
as well as formulate similar conjectures for non-toric surfaces. In
principle, for surfaces of general type, one can proceed by analogy
with \textcite{M_Donaldson}, using the master space constructed by
\textcite{Donaldson_alg_surfaces}.

On the physics side, several interesting open questions on Coulomb branch
localization in four-dimensional gauge theory emerge immediately from
\zcref{sec:4d, sec:localp2}:

\begin{itemize}
\item The main difficulty in the computations carried out in
\zcref{sec:localp2} is the rapidly growing combinatorial complexity of
Young diagrams.  It would be useful, in order to speed up computations,
to have a recursion relation for the instanton partition function
of our theory, such as those available for more conventional 4d
\cite{Poghossian:2009mk} and 5d \cite{Yanagida:2010vz} theories.

\item If one is only interested in the non-equivariant limit, one would
benefit from knowing the Seiberg--Witten curve of our theory.  Nekrasov's
conjecture \cite{Nekrasov:2002qd} for the K-theoretic partition function
on $\BC^2 \times S^1_\beta$ is the general statement that, given the
genus expansion
\begin{equation} \label{eq:nek-conj}
\ep_1 \ep_2 \log Z =
F_0 + (\ep_1+\ep_2) H + (\ep_1+\ep_2)^2 G + \ep_1 \ep_2 F_1 + \dots
\end{equation}
$\log Z$ is regular in the $\ep_{1,2} \to 0$ limit. 
If $S$ is compact, this implies that we are allowed to turn off equivariant
parameters in \zcref{eq:zk}: for any given $\xi$, the integrand
has a well-defined limit $q_1,q_2 \to 1$,
\begin{multline}
\lim \log (z_{cl} \cdot z_{pert} \cdot z_{inst}) =
\int_S F_0 (\anew \eu^{-\xi}) + \int_S c_1 (S) H (\anew \eu^{-\xi}) \\
+ (2\chi(S)+3\sigma(S)) G(\anew) + \chi(S) F_1(\anew)
\end{multline}
which can be expressed in terms of four universal functions, and $Z_r^C$
can be written as a contour integral of that, the contour being the
limit of our equivariant contour.  (In this limit, some simple poles
may collide, resulting in higher-order poles.)
A related question is whether the Coulomb branch partition
function admits an alternative construction using only the infrared limit
of the gauge theory \cite{Monopoles_4_manifolds, Int_C_branch,Uplane,
Top_corr}.

\item The previous point is also amenable to generalization to non-toric
surfaces, where we only use equivariance with respect to the fiberwise
antidiagonal torus action on $X$.

\item The extension to higher-rank stable pair invariants of
local Calabi--Yau fourfolds is also possible in principle, albeit
computationally more challenging.


\item Including the fermionic surpermultiplets $(\Upsilon, \Psi)$ the
gauge theory studied in \zcref{sec:4d} is superconformal. Hence
one could expect some modular properties for its partition function.
\end{itemize}

\subsection{Organization}

The paper is structured as follows: \zcref{sect:SP_ADHM} consists of a
detailed treatment of the spectral correspondence between stable pairs
and ADHM sheaves, \zcref{sec:4d} provides a thorough account of Coulomb
branch localization, while \zcref{sec:localp2} is focused on the case
of a local projective plane.

\subsection{Acknowledgments}
We are very grateful to Mauro Porta and Francesco Sala for very helpful
discussions on the material in \zcref{sect:SP_ADHM},
to Wu-yen Chuang for collaboration on a related project, and to Martijn
Kool for many inspiring discussions, and for convincing us
that counting surfaces in fourfolds is an interesting problem.

We thank Francisco Morales, Mauricio Romo and Massimiliano Ronzani for
discussions on 4d gauge theories.

The research of NP was supported by the US Department of Energy
under grant DE-SC0010008.

\section{Stable pairs on fourfolds as ADHM sheaves} \label{sect:SP_ADHM} 

In this section $X$ is the total space of direct sum of line bundles
$\LL_1 \oplus \LL_2$ over a smooth projective surface $S$. We
will not assume that $H^1(S,\IQ)=0$ unless otherwise stated.  Moreover,
we also do not need to
assume the Calabi--Yau condition $\LL_1 \otimes \LL_2\simeq \omega_S$.
Proceeding by analogy with \cite{ADHMsheaves}, the goal of this section
is to establish a one-to-one correspondence between ${\PT}_1$-stable
pairs on $X$ with compact support and framed ADHM sheaves on $S$. The
main results are formulated as isomorphisms of derived stacks in
\zcref{thm:triplesADHM, cor:pairsADHM}.  While the general approach is
similar to \cite{ADHMsheaves}, the fourfold case is technically more
involved, hence this section will provide a detailed self-contained
account.

\subsection{ADHM sheaves}\label{sect:ADHM} 
Recall that the ADHM quiver is a quiver with relations defined by the diagram
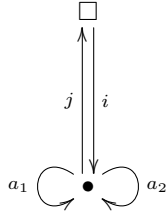
\[
\xymatrix{ 
& \Box  \ar@<0.5ex>[dd]^-{i} & \\
& & \\
& \bullet \ar@(ul,dl)_{a_1} \ar@(ur,dr)^-{a_2} 
 \ar@<0.5ex>[uu]^-{j} \\}
\]
where the ideal of relations in the path algebra is generated by 
\[
a_1a_2-a_2 a_1+  ij.
\]
A rank one framed cyclic representation of the ADHM is defined by a
quadruple of linear maps $(A_1, A_2,I,J)$, with
\[
A_i \in {\rm End}(\IC^n),\ 1\leq i\leq 2,\qquad 
I \in {\rm Hom}(\IC, \IC^n), \qquad  J \in {\rm Hom}(\IC^n,\IC),
\]
satisfying the quadratic relation $[A_1,A_2]+IJ=0$.  The cyclicity
condition states that the image of $I$ is not contained in any subspace
$0\subsetneq V' \subsetneq V$ preserved by $A_i$, $1\leq i \leq
2$. Isomorphisms are defined by the natural action of $GL(n)$. As shown
in \cite[Theorem 2.1]{Lectures_Hilb}, the moduli space of rank one framed
cyclic ADHM quiver representation is isomorphic to the Hilbert scheme of
$n$ points on $\IA^2$.  Moreover, \cite[Proposition 2.8]{Lectures_Hilb}
also shows that $J=0$ for all rank one stable framed representation. Hence
$A_1$, $A_2$ are commuting endomorphism.

For us, an ADHM sheaf on $S$ will be defined by a quintuple
$(\calL, \calE, \Phi_1, \Phi_2, \phi)$, where:
\begin{itemize}
\item $\calL$ is a line bundle on $S$ with $\ch(\calL)=\ch(\calO_S)$,  
\item $\calE$ is a torsion free coherent sheaf on $S$, and 
\item $\Phi_i: \calE \to \calE \otimes \LL_{i}$, $1\leq i\leq 2$, and $\phi: \calL \to \calE$
are morphisms of coherent sheaves on $S$ satisfying the relation
\[
(\Phi_1 \otimes \mathbf{1}_{\LL_2}) \circ \Phi_2
- (\Phi_2 \otimes \mathbf{1}_{\LL_1}) \circ \Phi_1 =0.\]
\end{itemize}
In addition we will also impose a stability condition  which requires
the data $(\calE, \Phi_1, \Phi_2, \phi)$ to be generically cyclic. More
precisely, a torsion free ADHM sheaf $(\calL, \calE, \Phi_1, \Phi_2,
\phi)$ will be called stable, or generically cyclic, if the image of
$\phi$ is not contained in any proper saturated subsheaf $0 \subsetneq
\calE' \subsetneq \calE$ preserved by $\Phi_i$, $1\leq i \leq 2$. For
completeness, recall that a subsheaf $\calE'\subset \calE$ is said to
be saturated if the quotient $\calE/\calE'$ is a torsion free sheaf.

A framed ADHM sheaf on $S$ will be an ADHM sheaf as above where
$\calL=\calO_S$ is fixed. We will denote such objects by $( \calE,
\Phi_1, \Phi_2, \phi)$.

Proceeding by analogy with \cite{ADHMsheaves}, the goal of this section
is to establish a one-to-one correspondence between $\PT_1$-stable
pairs on $X$ with compact support and stable ADHM sheaves on $S$. A
precise statement is formulated in \zcref{thm:triplesADHM} in the form
of an isomorphism of derived moduli stacks.  While the general approach
is similar to \cite{ADHMsheaves}, some technical details will be more
involved since in the present context ADHM sheaves are not necessarily
locally free.

\subsection{Beilinson Theorem}\label{sect:Beilinson}
We first construct a projective completion 
\[
\xymatrix{ 
X \ar[rr] \ar[dr]_-{\pi} & & Y \ar[dl]^-{\rho} \\
& S & \\
}
\]
over $S$ as a projective bundle $Y \coloneq \operatorname{Proj}(\LL)$,
where $\LL \coloneq \LL_0^{-1} \oplus \LL_1^{-1} \oplus \LL_2^{-1}$,
and $\LL_0 \coloneq \calO_S$.

Let $\calO_Y(1)$ be the relative tautological line bundle, which satisfies
\[ 
\rho_*\calO_Y(1)=\LL_0^{-1} \oplus \LL_1^{-1} \oplus \LL_2^{-1}.
\]
Let $z_i \in H^0 (\rho^* \LL_i \otimes \calO_Y(1))$, $0\leq i \leq 2$,
be the tautological sections satisfying $\rho_*(z_i)=1$.  Note that $X$
coincides with the complement of the divisor $z_0=0$ in $Y$, which will
be denoted by $D_\infty$.

We next recall Beilinson theorem for the derived category of the
projective plane \cite{Beilinson}. Let
\[
T = \calO_{\IP^2}\oplus \calO_{\IP^2}(1) \oplus \calO_{\IP^2}(2),
\]
and let 
\[
B = {\rm RHom}_{\IP^2}(T, T),
\]
which has a natural algebra structure given by composition.  Then the
assignment
\[
E \mapsto {\rm RHom}_{\IP^2}(T, E)
\]
determines an equivalence of derived categories 
\[
D^b(\IP^2) \to D^b(B^{\sf op}-{\rm mod}) 
\]
where the target is the bounded derived category of $B^{\sf op}$-modules.

In the present framework, we will use the relative variant of Beilinson's
theorem for the projective bundle $\rho:Y \to S$. In the relative setting,
the above result was first stated with proof in \cite{Beilinson}. It
was proven in \cite[Theorem 4.1]{Perverse_CM} for non-commutative
$\IP^1$-bundles over curves, but their arguments clearly apply for more
general projective bundles.  One can also deduce the relative variant of
Beilinson's theorem from the decomposition theorem proven in \cite[Theorem
2.6]{Proj_bundles}, reasoning by analogy to \cite{Beilinson}.

Let
\[
\calT \coloneq \calO_Y \oplus \calO_Y(1) \oplus \calO_Y(2)
\]
and let 
\[
\calB \coloneq R\rho_*R{\mathcal Hom}_Y(\calT, \calT). 
\]
Then note that $\calB$ is a locally free complex concentrated in degree
zero,
\begin{equation}
\calB \simeq \calO_S \oplus \calO_S \oplus \calO_S 
\oplus \LL \oplus \LL \oplus S^2(\LL),
\end{equation} 
and has a natural sheaf of algebras structure given by composition.  Let
${\calB}^{\sf op}-{\rm mod}$ denote the abelian category of quasi-coherent
$\calB^{\sf op}$-modules, and let $D^b(\calB^{\sf op}-{\rm mod})$ be
the associated derived category of complexes with coherent cohomology.
Then we have:
\begin{theorem}\label{thm:tilting}
The assignment 
\be\label{eq:tiltingA}
\calF \mapsto R\rho_*R{\mathcal Hom}_Y(\calT, \calF) 
\ee
determines an equivalence of derived categories
\be\label{eq:dereq}
\tau: D^b(Y) \to D^b({\calB}^{\sf op}-{\rm mod}). 
\ee
\end{theorem}

By definition, a $\calB^{\sf op}$-module consists of a 
triple $(\calE_0, \calE_1, \calE_2)$ of quasi-coherent
$\calO_S$-modules together with morphisms 
\[
a: \calE_0 \to \calE_1 \otimes (\LL_0 \oplus \LL_1 \oplus \LL_2), 
 \qquad b: \calE_1 \to \calE_2 \otimes (\LL_0 \oplus \LL_2 \oplus \LL_1)
\]
satisfying the quadratic relations 
\[
 (b_i \otimes \mathrm{id}_{\LL_j}) \circ a_j =
 (b_j \otimes \mathrm{id}_{\LL_i}) \circ a_i
\]
for $0 \leq i, j \leq 2$, $i\neq j$.  A $\calB^{\sf op}$-module will
be called coherent if the $\calO_S$-modules $\calE_i$, $0\leq i\leq 2$,
are coherent.

Now let ${\rm Coh}_{\sf c}(X)$ be the abelian category of coherent
sheaves on $X$ with compact support, or, equivalently, coherent sheaves
on $Y$ with set theoretic support contained in $X$. For simplicity, the
objects of ${\rm Coh}_{\sf c}(X)$ will be called horizontal sheaves in
the following.

We also introduce a category of Higgs sheaves on $S$, consisting of
triples $(\calE, \Phi_1, \Phi_2)$ where $\calE$ is a coherent sheaf
on $S$ and $\Phi_i: \calE \to \calE \otimes \LL_i$, $1\leq i \leq 2$,
are morphisms of sheaves satisfying the relation 
\[
(\Phi_1 \otimes \mathbf{1}_{\LL_2}) \circ \Phi_2
- (\Phi_2 \otimes \mathbf{1}_{\LL_1}) \circ \Phi_1=0.
\]
Such triples form an abelian category ${\rm Higgs}(S;\LL_1, \LL_2)$.

Any Higgs sheaf $(\calE, \Phi_1, \Phi_2)$ determines canonically a
coherent ${\calB}^{\sf op}$-module with $\calE_i=\calE$, $0\leq i \leq
2$, and
\[
a_0=b_0 = \mathrm{id}_\calE, \qquad a_i = b_i = \Phi_i, \ 1\leq i \leq 2.
\]
It is straightforward to check that this construction determines a functor 
\be\label{eq:higgsfunct} 
{\rm Higgs}(S; \LL_1, \LL_2) \to {\mathcal B}^{\sf op}-{\rm mod}.
\ee
Moreover, note that the coherent $\calB^{\sf op}$-modules
$(\calE_i, a_i, b_i)$, $0\leq i\leq 2$, with
$a_0$ and $b_0$ isomorphisms form a full abelian subcategory $\calC \subset
\mathcal{B}^{\sf op}-{\rm mod}$.  Then it is also straightforward to
check the following:
\begin{proposition}\label{prop:higgscat} 
The functor \eqref{eq:higgsfunct} induces an equivalence of
Abelian categories ${\rm Higgs}(S; \LL_1, \LL_2) \to \calC$.
\end{proposition}

Now let ${\rm Coh}^{(2)}_{\sf c}(X)\subset {\rm Coh}_{\sf c}(X)$ be
the full subcategory consisting of purely two dimensional sheaves.
Let ${\mathcal Higgs}_{\sf t.f.}(S;\LL_1, \LL_2)\subset {\mathcal
Higgs}(S;\LL_1, \LL_2)$ be the full subcategory consisting of torsion
free Higgs sheaves.  Our next goal is to show that that the tilting
equivalence \eqref{eq:dereq} induces a natural categorical equivalence
${\rm Coh}^{(2)}_{\sf c}(X)\to {\mathcal Higg}_{\sf t.f.}(S;\LL_1, \LL_2)$
using \zcref{prop:higgscat}.  A similar result was proven in \cite{TT_VW}
for total spaces of line bundles on surfaces. We will provide a detailed
proof below because the present case is more involved.  Moreover, an
important fact for us is to establish that this categorical equivalence
is induced by a derived equivalence.

First note that the scheme theoretic support of any sheaf $\calF \in  {\rm
Coh}^{(2)}_{\sf c}(X)$ is by definition disjoint from $D_\infty\subset
Y$. Since the latter is relatively ample, we have:
\begin{lemma}\label{lem:supplemm}
Let $\calF\in {\rm Coh}_{\sf c}(X)$ and let $Z\subset X$ be its scheme-theoretic support. 
Then the induced projection $Z\to S$ has zero dimensional scheme-theoretic fibers. 
\end{lemma}
Using \cite[Theorem 18.8.5]{FOAG}, this further yields:
\begin{corollary}\label{cor:dirimA}
Let $\calF\in {\rm Coh}_{\sf c}(X)$. Then all higher direct images
$R^i\pi_*\calF$ are identically zero. In particular, any exact sequence
\[
0\to \calF_1 \to \calF_2 \to \calF_3 \to 0
\]
in $\calF\in {\rm Coh}_{\sf c}(X)$ induces an exact sequence 
\[
0\to \pi_* \calF_1\to \pi_*\calF_2 \to \pi_*\calF_3 \to 0 
\]
in ${\rm Coh}(S)$. 

Moreover, $\pi_*\calF=0$ if and only if  $\calF=0$. 
\end{corollary} 

Given a sheaf $\calF \in {\rm Coh}_{\sf c}(X)$, we denote by $z_i \otimes
{\rm id}_\calF: \calF \to \calF\otimes \pi^*\LL_i$ the multiplication map
by the tautological section $z_i \in H^0(\calO_Y(1)\otimes \pi^*\LL_i)$.
As a consequence of \zcref{cor:dirimA}, the image of $\calF$ via the
tilting functor \eqref{eq:tiltingA} will be the ${\calB}^{\sf op}$-module
$(\calE_i, a_i, b_i)$, where $\calE_i = \pi_*(\calF \otimes \calO_Y(-i))$
and
\[
a_i = \pi_*(z_i\otimes {\rm id}_{\calF(-2)}), \qquad 
b_i = \pi_*(z_i \otimes {\rm id}_{\calF(-1)} )
\]
for $0\leq i \leq 2$. In particular, since the support of $\calF$ is
contained in $X$, $a_0$ and $b_0$ are isomorphisms. Hence $\calF$ belongs
to the subcategory $\calC$.  In conclusion, using \zcref{prop:higgscat},
the above construction yields a functor
\be\label{eq:cohtohiggs} 
h:{\rm Coh}_{\sf c}(X)\to {\rm Higgs}(S; \LL_1, \LL_2).
\ee

Conversely, given Higgs sheaf  $\calH \coloneq (\calE, \Phi_1, \Phi_2)$ on $S$,
let ${\bf M}(\calH)$ be the  monad complex
\be\label{eq:monadA}
\pi^*(\calE\otimes \LL_1^{-1} \otimes \LL_2^{-1}) \otimes \calO_Y(-2)
 \xrightarrow{d_{-1}} \pi^*(\calE \otimes (\LL_1^{-1}\oplus \LL_2^{-1})) \otimes 
\calO_Y(-1) \xrightarrow{d_0} \pi^*\calE,
\ee
of amplitude $[-2, \ 0]$, where the maps are given by 
\be\label{eq:monadB} 
d_{-1} = \left(\begin{array}{c} -z_2 +z_0 \pi^*\Phi_2 \\ 
z_2-z_0\pi^*\Phi_1 \end{array}\right)
\qquad 
d_0 = \left(z_1 - z_0 \pi^*\Phi_1 \ \
z_2-z_0\pi^*\Phi_2  \right).
\ee
\begin{lemma}\label{lem:monadA} 
Assume that $\calE$ is a nonzero coherent locally free sheaf on $S$.
Then the following hold:
\begin{itemize}
\item[$(i)$] The cohomology of ${\bf M}(\calH)$ is concentrated in degree zero and the
$0$-th cohomology sheaf $\calF$ is a purely two dimensional reflexive 
sheaf in ${\rm Coh}_{\sf c}(X)$. 
\item[$(ii)$] One has a natural isomorphism of Higgs sheaves $\calH\to h(\calF)$.
\end{itemize}
\end{lemma} 

\begin{proof} 
For $1\leq i\leq 2$, let $\Delta \subset Y$ be the effective  divisor defined by 
\[
 \det (z_1- z_0 \pi^*\Phi_1) \det(z_2- z_0\pi^*\Phi_2)=0
\]
and let $U = Y \setminus \Delta_{\sf red} $.  Then the restriction
$(d_{-1})_y$ to any point $y \in U$ is injective. Since the target and
the domain of $d_{-1}$ are locally free, this implies that $d_{-1}$
is injective.  Moreover, $\cok(d_{-1})$ is a torsion free sheaf on $Y$.

The middle cohomology sheaf of ${\bf M}(\calH)$ is a subsheaf of
${\cok}(d_{-1})$, hence it has to be torsion-free. At the same time,
given the explicit expressions of the differentials, it is straightforward
to check that the middle cohomology sheaf has to be set theoretically
supported on $\Delta_{\sf red}$.  Therefore it must be identically zero.
In conclusion the cohomology of ${\bf M}(\calH)$ is concentrated in
degree zero.

Let $\calF$ denote the degree zero cohomology sheaf of ${\bf
M}(\calH)$. One immediately checks that $\ch_k(\calF)=0$ for $0\leq
k \leq 1$, hence $\calF$ has two dimensional support. Furthermore,
it is also straightforward to check that the restriction  ${\bf
M}(\calH)|_{D_\infty}$ is exact, hence the support of $\calF$ is contained
in $X$. Finally, by construction, $\calF$ has a three term locally free
resolution, which implies that it is purely two dimensional and reflexive
by \cite[Proposition 1.1.10]{HL_moduli}.

The second claim is tautological. First note that the canonical projection
${\bf M}(\calH)\to \calF$ is a quasi-isomorphism, and $R\pi_*{\bf
M}(\calH)\simeq \calE$.  Therefore, by pushforward, one obtains an
isomorphism $\calE \to \pi_*\calF$.  Moreover, the map $z_i \otimes
{\rm id}_\calF: \calF \to \calF \otimes \calO_Y(1) \otimes \pi^*\LL_i$,
fits in the commutative diagram
\[
\xymatrix{ 
{\bf M}(\calH) \ar[rr]
\ar[d]_-{ z_i\otimes {\rm id}_{{\bf M}(\calH)}} 
&&  \calF \ar[d]
\ar[d]^-{z_i \otimes {\rm id}_\calF} 
\\
{\bf M}(\calH)\otimes \calO_Y(1)\otimes \pi^*\LL_i \ar[rr]&& \calF \otimes \calO_Y(1)\otimes \pi^*\LL_i.
\\
}
\]
Given the expression of $d_0$, by pushforward, one obtains the following
commutative diagram
\[
\xymatrix{ 
\calE \ar[rr]\ar[d]_-{ \Phi_i} &&  \pi_*(\calF)
\ar[d]^-{\pi_*(z_i \otimes {\rm id}_\calF)} \\
\calE\otimes \LL_i \ar[rr]&& \pi_*(\calF)\otimes \LL_i,
}
\]
where the horizontal arrows are isomorphisms. 
\end{proof}

Next note that the monad construction is functorial i.e.\ any morphism
of Higgs sheaves on $S$ induces a morphism between the associated monad
complexes, which induces morphisms between the associated cohomology
sheaves.  Then we have:
\begin{lemma}\label{lem:monadB}
Let $\calH_1 \to \calH_2$ be an injective morphism of locally free
Higgs sheaves on $S$.  Then the induced morphism $f:\calF_1 \to \calF_2$
is also injective.
\end{lemma} 

\begin{proof}
For $1\leq i \leq 2 $, let $\calE_i$ be the underlying locally free sheaf
of $\calH_i$.  Let $e:\calE_1 \to \calE_2$ be the associated morphism
of sheaves on $S$, which is compatible with the Higgs sheaf structure. 
Note the tautological commutative diagram
\[
\xymatrix{ 
\calE_1 \ar[r]^-{e} \ar[d] & \calE_2 \ar[d]\\
\pi_*\calF_1 \ar[r]^-{\pi_*(f)} & \pi_*\calF_2 ~\\
}
\]
where the vertical arrows are the canonical isomorphisms provided by part
ii of \zcref{lem:monadA}.  By pushforward, this yields an exact sequence
\[
0\to \pi_*(\calK) \to \calE_1 \xrightarrow{e} \calE_2
\]
where $\calK \coloneq \ker(\calF_1 \to \calF_2)$. Since $\calK$ belongs
to ${\rm Coh}_{\sf c}(X)$, it must be identically zero by \zcref{cor:dirimA}.
\end{proof} 

\begin{lemma}\label{lem:monadC} 
Let $\calH$ be a nonzero coherent torsion free Higgs sheaf on $S$. 
Then the cohomology of the associated monad complex ${\bf M}(\calH)$ 
is concentrated in degree zero. Moreover, the degree zero cohomology sheaf 
$\calF$ is purely two dimensional and belongs to ${\rm Coh}_{\sf c}(X)$. 
\end{lemma} 

\begin{proof} 
Let ${\calH}_{-1}\to {\calH}_{0} \to \calH$ be a locally free resolution
in the category of Higgs sheaves on $S$. Then we have an exact triangle
\[
{\bf M}({\calH}_{-1})\to {\bf M}({\calH}_{0}) \to {\bf M}(\calH)
\]
in $D^b(Y)$. Given \zcref{lem:monadA, lem:monadB}, the associated
long exact sequence shows that the cohomology of ${\bf M}(\calH)$ is
concentrated in degree zero. Moreover, we have an exact sequence of
degree zero cohomology sheaves
\[
0\to \calF_{-1}\to \calF_{0} \to \calF\to 0. 
\]
Since $\calF_{-1}$ and $\calF_0$ are purely two-dimensional and reflexive,
this implies that $\calF$ is purely two dimensional via \cite[Proposition
1.10]{HL_moduli}.
\end{proof} 

Finally, by analogy with \zcref{lem:monadB}, we also have: 
\begin{lemma}\label{lem:monadD}
Let 
\be\label{eq:Hexseq}
0 \to \calH_1 \to \calH_2 \to \calH_3 \to 0 
\ee
be an exact sequence of torsion free Higgs sheaves on $S$.  Then the
degree zero cohomology sheaves of the associated monad complexes fit in
an exact sequence 
\be\label{eq:Fexseq}
0 \to \calF_1 \to \calF_2 \to \calF_3 \to 0 
\ee
on $Y$. Moreover, the image of the complex \eqref{eq:Fexseq} via
the functor \eqref{eq:cohtohiggs} is tautologically isomorphic to
\eqref{eq:Hexseq}.
\end{lemma}

Now let $\calF$ be a purely two dimensional coherent sheaf on $Y$ with
compact support contained in $X$.  We will prove below that the direct
image $\rho_*\calF$ is a torsion free $\calO_S$-module.

In order to fix notation, the dual of a coherent sheaf $\calF$ of
dimension $d$ on $Y$ will be defined by
\[
\calF^* \coloneq {\mathcal Ext}^{4-d}(\calF, \calO_Y),
\]
as in \cite[Definition 1.1.7]{HL_moduli}.  The double dual of $\calF$ will
be denoted by $\calF^{**}$.  As shown in \cite[Lemma 1.1.8]{HL_moduli},
one has a natural morphism $\theta: \calF \to \calF^{**}$. Moreover,
$\theta$ is injective if $\calF$ is purely $d$-dimensional by
\cite[Proposition 1.1.10]{HL_moduli}.  Therefore, for any sheaf $\calF$
of pure dimension, we have an exact sequence 
\be\label{eq:FseqA} 
 0 \to \calF \xrightarrow {\theta} \calF^{**} \to \calG \to 0.
\ee

\begin{lemma}\label{lem:FlemmA} 
Let $\calF$ be a purely two dimensional coherent sheaf on $Y$.  Then
$\calG$ is zero dimensional. Moreover, $\calF^{**}$ admits a three term
locally free resolution.
\end{lemma} 
 
\begin{proof} 
Let $\calF^\vee \coloneq R{\mathcal RHom}_Y(\calF, \calO_Y)$ be the derived
dual of $\calF$.  Then one has
\[
\ch_k(\calF^\vee) = (-1)^k \ch_k(\calF)
\]
for all $0\leq k \leq 4$.  By \cite[Propositions 1.1.6 and 1.1.10]{HL_moduli}, one
has ${\mathcal Ext}^k(\calF, \calO_Y)=0$ for all $k \neq 2,3$. Moreover,
${\mathcal Ext}^3(\calF, \calO_Y)$ is zero-dimensional.  This implies that
\[
\ch_k(\calF^*) = (-1)^k \ch_k(\calF)
\]
for $0\leq k \leq 3$. By \cite[Proposition 1.1.10]{HL_moduli}, the
$\calO_Y$-module $\calF^*$ is reflexive, and ${\mathcal Ext}^k_Y(\calF^*,
\calO_Y) =0$ for all $k \neq 2$. Repeating the above argument, this
further implies that
\[
\ch_k(\calF^{**}) = \ch_k(\calF)
\]
for $0\leq k \leq 3$. Then $\calG$ is indeed zero-dimensional. 

The second part of \zcref{lem:FlemmA} is proven in \cite[Proposition
1.1.10]{HL_moduli}
\end{proof}

The next goal is to show that $\calF^{**}$ is flat over $S$. First, we have:

\begin{lemma}\label{lem:dirimA}
Let 
\[
0\to {\sf F}_{-2} \to  {\sf F}_{-1} \to {\sf F}_0 \to \calF^{**} \to 0 
\]
be a locally free resolution of $\calF^{**}$. Let $I \subset {\sf F}_0$
be the image of the morphism ${\sf F}_{-1} \to {\sf F}_0$.  Let $p\in S$
be an arbitrary point and let $Y_p$ be the scheme theoretic fiber of
$\rho$ at $p$.  Then $I$ is flat over $S$ and $I \otimes \calO_{Y_p}$
is a torsion free $\calO_{Y_p}$-module.
\end{lemma} 

\begin{proof} 
By construction, $I$ has a locally free resolution 
\[
0\to {\sf F}_{-2} \to {\sf F}_{-1} \to I \to 0 
\]
which yields an injective morphism 
\[
{\mathcal Tor}_1^Y(I, \calO_{Y_p}) \to {\sf F}_{-2} \otimes \calO_{Y_p}.
\]
By \zcref{lem:supplemm}, the scheme-theoretic support of ${\calF}^{**}$
is a closed subscheme $Z \subset Y$ so that the induced projection $Z\to
S$ has zero-dimensional fibers.  Moreover, the exact sequence
\[
0\to I \to {\sf F}_0\to {\calF}^{**}\to 0 
\]
shows that $I$ is locally free on the complement $Y\setminus {Z}_{\sf
red}$. Hence ${\mathcal Tor}_1^Y(I, \calO_{Y_p})$ is set theoretically
supported on $Z_{\sf red} \cap Y_p$, which is zero dimensional by
\zcref{lem:supplemm}.  Since ${\sf F}_{-2} \otimes \calO_{Y_p}$ is
locally free on $Y_p$, and $Y_p$ is smooth, this implies that
\begin{align}\label{eq:torvanishing}
{\mathcal Tor}_1^Y(I, \calO_{Y_p}) =0.
\end{align} 

By \cite[Lemma 3.1.1(b)]{TB_thesis}, the above vanishing result implies
that $I$ is flat over $S$.  Hence one has an exact sequence
\[
0 \to {\sf F}_{-2} \otimes \calO_{Y_p} \to {\sf F}_{-1} \otimes \calO_{Y_p} \to 
 I \otimes \calO_{Y_p}\to 0.
\]
Since the first two terms are locally free, the associated long exact
sequence shows that
\be\label{eq:extvanishingB} 
 {\mathcal Ext}^2_{Y_p}( I \otimes \calO_{Y_p}, \calO_{Y_p}) =0. 
\ee
Moreover, $I \otimes \calO_{Y_p}$ is a  locally free $\calO_{Y_p}$-module
on the complement of the set theoretic intersection $Z_{\sf red}
\cap Y_p$, which is zero dimensional. Then the maximal torsion
$\calO_{Y_p}$-submodule $\calQ\subset I \otimes \calO_{Y_p}$ is
zero-dimensional. Therefore we have an exact sequence
\[
 0\to \calQ \to  I \otimes \calO_{Y_p}\to \calG \to 0
\]
where $\calG$ is a torsion-free $\calO_{Y_p}$-module.  Since $Y_p$
is a smooth surface, we have
\[
 {\mathcal Ext}^2_{Y_p}(\calG, \calO_{Y_p}) =0.
\]
Then the associated long exact sequence yields an isomorphism 
\[
{\mathcal Ext}^2_{Y_p}(I \otimes \calO_{Y_p}, \calO_{Y_p}) \to {\mathcal Ext}^2_{Y_p}(\calQ, 
\calO_{Y_p}).
\]
Given the vanishing result \eqref{eq:extvanishingB}, this implies that
$\calQ=0$. In conclusion, $I\otimes \calO_{Y_p}$ is a torsion-free
$\calO_{Y_p}$-module.
\end{proof}

Using the same notation as in \zcref{lem:dirimA}, we next prove: 
\begin{lemma}\label{lem:dirimB} 
$\calF^{**}$ is flat over $S$. 
\end{lemma} 

\begin{proof} 
By \cite[Lemma 3.1.1(b)]{TB_thesis}, it suffices to prove that 
\[
{\mathcal Tor}_1^Y(\calF^{**}, \calO_{Y_p}) =0 
\]
for any $p\in S$.  The exact sequence 
\[
0 \to I \to {\sf F}_0 \to \calF'\to 0
\]
yields an injective morphism 
\[
 {\mathcal Tor}_1^Y(\calF^{**}, \calO_{Y_p}) \to I \otimes \calO_{Y_p}~. 
\]
Since $\calF^{**}$ is scheme theoretically supported on $Z$, the set
theoretic support of $ {\mathcal Tor}_1^Y(\calF^{**}, \calO_{Y_p})$
is contained in the set theoretic intersection $Z_{\sf red}\cap
Y_p$, which is zero dimensional by \zcref{lem:supplemm}.  Moreover,
$I \otimes \calO_{Y_p}$ is a torsion free $\calO_{Y_p}$-module by
\zcref{lem:dirimA}. This implies the required vanishing.
\end{proof} 

In conclusion, by \zcref{lem:FlemmA, lem:dirimB}, one obtains an exact
sequence
\be\label{eq:FseqAB} 
0 \to \calF \xrightarrow{\theta} \calF^{**}\to \calG \to 0 
\ee
where the middle term is reflexive and flat over $S$, and $\calG$ is
zero-dimensional.  Moreover, the first two terms are scheme theoretically
supported on a closed subscheme $Z\subset Y$, so that the induced
projection $Z\to S$ has zero-dimensional fibers.
\begin{lemma}\label{lem:dirimC}
One has an exact sequence of $\calO_S$-modules
\be\label{eq:dirimA} 
0 \to \rho_*\calF \to \rho_* \calF^{**} \to \rho_*\calG \to 0 
\ee
where the middle term is locally free and the third is zero-dimensional.
In particular, the direct image $\rho_*\calF$ is torsion-free.
\end{lemma}

\begin{proof} 
Under the stated conditions, the exact sequence \eqref{eq:dirimA}
follows from \zcref{cor:dirimA}.  Moreover, one has
\[
H^k(\calF^{**}\otimes \calO_{Y_p}) =0 
\]
for all $k \geq 1$, for any point $p\in S$. Since $\calF^{**}$ is flat
over $S$, the base change theorem implies that $\rho_*\calF^{**}$ is
locally free. Finally, $\rho_*\calG$ is zero dimensional since $\calG$
is zero-dimensional.
\end{proof} 
In conclusion, \zcref{lem:monadC, lem:monadD, lem:dirimC} yield:
\begin{proposition}\label{prop:spectral} 
The assignment $\calF \mapsto h(\calF)$ determines a categorical
equivalence ${\rm Coh}_{\sf c}^{(2)}(X) \to {\rm Higgs}_{\sf t.f.}(S;
\LL_1, \LL_2)$ which preserves exact sequences. Its inverse maps a Higgs
sheaf $\calH$ to the degree zero cohomology sheaf of the associated monad
complex ${\bf M}(\calH)$. Moreover, $\calH$ is locally free if and only
if $\calF$ is reflexive.
\end{proposition} 

\begin{remark}
Using analogous arguments, one can prove that the functor
\eqref{eq:cohtohiggs} is an equivalence of abelian categories.
The details will be omitted because this result will not be used in
the following. 
\end{remark}

\subsection{Stable pair moduli spaces}\label{sect:pairs_moduli} 

This section summarizes some of the foundational results proven in
\cite{St_pairs_fourfolds, Counting_surfaces_I, Counting_surfaces_II} on
moduli spaces of $\PT_1$-stable pairs.  In order to make a connection with
moduli spaces of objects in the derived category $D^b(Y)$, we will work
with some closely related objects called ${\rm PT}_1$-stable triples. A
${\rm PT}_1$-stable triple on $Y$ will be defined by the data $(\calL,
\calF, s)$ where
\begin{itemize} 
\item $\calL$ is a line bundle on $Y$ with $\ch(\calL) = \ch(\calO_Y)$, 
\item $\calF$ is a purely two-dimensional sheaf on $Y$, and 
\item $s: \calL \to \calF$ is morphism of sheaves on $Y$ so that
$\cok(s)$ has one-dimensional support.
\end{itemize} 
We define the support of a triple $(\calL, \calF, s)$ to be the set
theoretic support of $\calF$.  A triple will be said to be supported in
$X$ if its support is contained in $X$.

Let $ST_{1}(Y)$ be the moduli stack of stable triples on $Y$.
This is identified with a moduli stack of objects in
the derived category $D^b(Y)$ as follows.

Let $\tau$ be the $t$-structure on $D^b(Y)$ defined by the torsion pair 
\[
\left({\rm Coh}_{\leq 2}(Y), \ {\rm Coh}_{\geq 3}(Y)\right),
\]
where ${\rm Coh}_{\leq 2}(Y)\subset {\rm Coh}(Y)$ is the full-subcategory
of coherent sheaves with at most two-dimensional support, and ${\rm
Coh}_{\geq 3}(Y)\subset {\rm Coh}(Y)$ is its orthogonal complement.
Let $\coh(Y,\tau)$ be the moduli stack of objects in the heart of the
above $t$-structure, which is an algebraic stack of locally finite type.
Then by \cite[Theorem 2.19]{St_pairs_fourfolds} and \cite[Theorem
1.4]{Counting_surfaces_I}, the moduli stack of stable triples $ST_1(Y)$
is naturally isomorphic to an open substack of $\coh(Y,\tau)$ determined
by a polynomial stability condition.

Moreover, note that the above torsion pair is open in the sense
of \cite{Mod_K_trivial} i.e.\ the moduli stacks $Coh_{\leq 2}(Y)$
and $Coh_{\geq 3}(Y)$ are open in $Coh(Y)$.  Then, by analogy with
\cite[Section 2.1]{Two_dim_Hall}, the moduli stack ${\bf Perf}_\tau(Y)$
admits a derived enhancement, which is a substack of the derived
moduli stack of perfect complexes ${\bf Perf}(Y)$ constructed in
\cite{Mod_obj_dg}.  Therefore, as an open substack of $\coh(Y,\tau)$,
the moduli stack $ST_1(Y)$ also admits a derived enhancement ${\bf
ST}_1(Y)$. Furthermore, by analogy with \cite[Section 2.w]{Two_dim_Hall},
the derived stack ${\bf ST}_1(Y)$ admits a tangent complex, which
coincides with the restriction of the tangent complex of ${\bf Perf}(Y)$
to ${\bf ST}_1(Y)$.

One also has a determinant map 
\be \label{eq:detmap} 
\det: {\bf ST}_1(Y) \to {\bf Pic}_0(Y),
\ee 
where the target is the derived Picard stack of $Y$ consisting of line
bundles with the same Chern character as $\calO_Y$.

Then the derived moduli stack of $\PT_1$-stable pairs on $Y$
is defined through the pull-back square 
\[
\xymatrix{ 
{\bf SP}_1(Y)\ar[r] \ar[d] & {\bf ST}_1(Y)\ar[d]^-{\det} \\
{\rm Spec}\, \IC\ar[r]& {\bf Pic}_0(Y), \\
}
\]
where the bottom horizontal arrow maps ${\rm Spec}\, \IC$ to $[\calO_Y]$. 

Finally, \cite[Proposition
2.9]{Counting_surfaces_II}, shows that the support condition ${\rm
Supp}(\calF) \subset X$ is open. Therefore one also has an open derived
substack ${\bf ST}_1(X) \subset {\bf ST}_1(Y)$, respectively an open
derived subspace ${\bf SP}_1(X) \subset {\bf SP}_1(Y)$ whose closed points
are in one-to-one correspondence with objects supported in $X$. In the
latter case, the closed points are also in one-to-one correspondence with
stable pairs on $X$. As shown in \cite[Section 4.1]{St_pairs_fourfolds}
and \cite[Theorem 4.1]{Counting_surfaces_I}, the construction of
\cite{OT-isotropic} yields a  virtual fundamental cycle and a virtual
structure sheaf on classical truncation ${\bf SP}_1(X)$.

\subsection{From stable pairs to ADHM sheaves}\label{sect:pairs_ADHM} 
Our next goal is to derive an explicit connection between $\PT_1$-stable
pairs on $X$ and framed ADHM sheaves, using the tilting functor
\eqref{eq:tiltingA}.

First note the following result, which follows easily from the
basic properties of projective bundles. Recall that ${\bf Pic}_0(Y)$
is the derived Picard stack consisting of line bundles $\calL$ with
$\ch_1(\calL)= \ch_1(\calO_Y)$.  Let ${\bf Pic}_0(S)$ be its analogue
for line bundles on $S$. Then we have:
\begin{proposition}\label{prop:linebundles}
The tilting functor \eqref{eq:tiltingA} induces an isomorphism of
derived stacks 
\[
{\bf Pic}_0(Y) \to {\bf Pic}_0(S) 
\]
which maps a line bundle $\calL$ on $Y$ to $\rho_*(\calL)$.
\end{proposition}
As a consequence of \zcref{prop:linebundles}, the determinant map
\eqref{eq:detmap} determines a second map
\be\label{eq:detmapB} 
\mathrm{det}_S: {\bf ST}_1(Y) \to {\bf Pic}_0(S). 
\ee
Moreover, \zcref{prop:spectral} yields: 
\begin{lemma}\label{lem:stpairsX} 
The assignment 
\[ 
(\calL, \calF, s) \mapsto (\rho_*(\calL), h(\calF), \rho_*(s)) 
\]
determines an equivalence between the groupoid of stable triples on $Y$
with support in $X$ and the groupoid of stable ADHM sheaves on $S$.

Similarly, the assignment 
\[
(\calF,s) \mapsto (h(\calF), \pi_*(s)) 
\]
determines an equivalence between the groupoid of stable pairs on $X$
and the groupoid of framed stable ADHM sheaves on $S$.
\end{lemma} 

Now let ${\bf Perf}(\calB^{\sf op}-{\rm mod})$ be the derived moduli stack
of complexes in $D^b(\calB^{\sf op}-{\rm mod})$. Then \zcref{thm:tilting,
lem:stpairsX} yield:
\begin{theorem}\label{thm:triplesADHM} 
The tilting functor \eqref{eq:tiltingA} induces a morphism of derived
stacks ${\bf ST}_1(Y) \to {\bf Perf}(\calB^{\sf op}-{\rm mod})$
mapping ${\bf ST}_1(Y)$ isomorphically onto a derived substack
${\bf ADHM}(S, \LL_1,\LL_2)\subset {\bf Perf}(\calB^{\sf op}-{\rm mod})$
whose closed points are
in one-to-one correspondence with stable ADHM sheaves on $S$.  Moreover,
the determinant map \eqref{eq:detmapB} factors as follows
\[
\xymatrix{ 
{\bf ST}_1(Y) \ar[dr]_-{{\det}_S}  \ar[rr]^-{\sigma} && {\bf ADHM}(S,\LL_1, \LL_2)\ar[dl]^-{\varphi}\\
& {\bf Pic}_0(S) & \\
}
\]
where 
\begin{itemize} 
\item $\sigma$ is the induced isomorphism, and 
\item 
$\varphi$ is the forgetful morphism mapping an ADHM sheaf 
$(\calL, \calE, \Phi_1, \Phi_2, \phi)$ to $\calL$. 
\end{itemize}
\end{theorem} 

Furthermore, we define the derived moduli stack of framed ADHM sheaves on
$S$ as the scheme theoretic inverse image ${\bf ADHM}_{\sf fr}(S, \LL_1,
\LL_2) \coloneq \varphi^{-1}([\calO_S])$.  Note that its closed points are in
one-to-one correspondence with framed stable ADHM sheaves on $S$. Then,
as a consequence of \zcref{thm:triplesADHM}, we obtain:
\begin{corollary}\label{cor:pairsADHM} 
The isomorphism $\sigma$ induces an isomorphism of derived algebraic spaces
${\bf SP}_1(X) \to {\bf ADHM}_{\sf fr}(S,\LL_1, \LL_2)$.  
\end{corollary}

An immediate consequence of \zcref{cor:pairsADHM} is that the virtual
structure sheaf constructed in \cite[Section 4.1]{St_pairs_fourfolds}
and \cite[Theorem 4.1]{Counting_surfaces_I} define analogous structures on
the classical truncation ${}^{\sf cl}{\bf ADHM}_{\sf fr}(S,\LL_1, \LL_2)$.

\subsection{Topological invariants}\label{sec:topinv}
For future reference, in this section we collect some results on
topological invariants coherent sheaves $\calF$ on $Y$ with compact
support contained in $X$.  The topological invariants of any two
dimensional coherent sheaf $\calF$ on $Y$ are defined by triple
\[
(\ch_2(\calF), \ch_3(\calF), \chi(\calF)) \in H^4(Y,\IQ) \oplus H^6(Y,\IQ) \oplus \IZ.
\]
Since all homology and cohomology groups in this section will be defined
over $\IQ$, we will simply omit the coefficient ring in the following.
We will also use the symbol $H_c^\bullet$ for cohomology with compact
support.

Let $D_\infty \subset Y$ be the reduced closed complement of $X$. Note
that there is a canonical section $\sigma_\infty: S \to Y$ mapping $S$
isomorphically onto $D_\infty$.  Recall that $\sigma : S \to Y$ is the
unique section of $Y$ over $S$ which factors through the zero section
of $X$ over $S$. The latter will be denoted by $o:S\to X$.

\begin{proposition}\label{prop:topinvC} 
Assume that $H^{1}(S)=0$. Then there is a commutative diagram 
\be\label{eq:topdiag}
\xymatrix{ 
0 \ar[r] 
                   & H^{2k}_{c}(X) \ar[r]^-{i^c_*} \ar[d] 
                   & H^{2k}(Y) \ar[d] \ar[r]^-{\sigma_\infty^*}
                   & H^{2k}(S) \ar[r] \ar[d]
                   & 0 \\
0\ar[r] 
                   & H_{8-2k}(S) \ar[r]^-{\sigma_*} 
                   & H_{8-2k}(Y) \ar[r]  
                   & H_{4-2k}(S) \ar[r]
                   & 0\\
                   }
\ee
with exact rows, where the vertical arrows are isomorphisms induced by
Poincar\'e duality.
\end{proposition}

\begin{proof} 
By Poincar\'e duality we have a commutative diagram
{\small
\be
\xymatrix{
\cdots \ar[r] & H^{k-1}(D_\infty) \ar[r] \ar[d] 
                   & H^k_{c}(X) \ar[r]^-{i^c_*} \ar[d] 
                   & H^k(Y) \ar[d] \ar[r] 
                   & H^k(D_\infty) \ar[r] \ar[d] 
                   & \cdots \\
\cdots \ar[r] & H_{9-k}(Y, X)\ar[r] 
                   & H_{8-k}(X) \ar[r]^-{i^h_*} 
                   & H_{8-k}(Y) \ar[r]  
                   & H_{8-k}(Y,X) \ar[r] 
                   & \cdots
                   }
\ee
}%
where the rows are canonical excision exact sequences and all vertical
arrows are isomorphisms.  Since the odd degree Betti numbers of $S$
vanish by assumption, for any $0\leq k \leq 2$, we obtain a second
commutative diagram 
\be 
\xymatrix{ 
0 \ar[r] 
                   & H^{2k}_{c}(X) \ar[r]^-{i^c_*} \ar[d] 
                   & H^{2k}(Y) \ar[d] \ar[r]^-{\sigma_\infty^*}
                   & H^{2k}(S) \ar[r] \ar[d]
                   & 0 \\
0\ar[r] 
                   & H_{8-2k}(X) \ar[r]^-{i^h_*} 
                   & H_{8-2k}(Y) \ar[r]  
                   & H_{8-2k}(Y,X) \ar[r]
                   & 0\\
                   }
\ee
with exact rows, where the vertical arrows are duality isomorphisms. Note
also that the map $\sigma_{*}: H_{2k}(S) \to H_{2k}(Y)$ factors
naturally through the pushforward map $o_*:H_{2k}(S) \to H_{2k}(X)$
associated to the zero section of $X$ over $S$.  The latter is an
isomorphism since $X$ is the total space of a vector bundle over $S$.
Finally, by Property $(2)$ listed in \cite[Ch. 2.6]{Rep_cpx_geom}, the
relative homology group $H_{8-2k}(Y,X)$ is isomorphic to the Borel--Moore
homology group $H^{BM}_{8-2k}(X)$ and since $X$ is the total space of a
real rank-four vector bundle over $S$ the pull-back map $H_{4-2k}^{BM}(S)
\to H_{8-2k}^{BM}(X)$ is an isomorphism. Moreover, since $S$ is compact,
$H^{BM}_{4-2k}(S) \simeq H_{4-2k}(S)$.

In conclusion, one obtains diagram \eqref{eq:topdiag}.
\end{proof} 

\begin{remark}\label{rem:topinvE} 
Provided that $H^1(S)=0$, \zcref{prop:topinvC} shows that the
topological invariants used in this paper coincide with those employed
in \cite{Counting_surfaces_I, Counting_surfaces_II}, which are defined
as elements of $H_c^\bullet(X, \IQ)$.
\end{remark}

\begin{remark}\label{rem:topinvF}
Since the map $i_h^*$ in \zcref{prop:topinvC} is injective, the natural
pairing $H^k(X) \otimes H_k(X) \to \IQ$ determines a pairing 
\[
\langle\ \ , \ \ \rangle : H^k(X) \otimes {\rm Im}(i^h_*) \to \IQ 
\]
for any $k\in \IZ$, $0\leq k \leq 8$. 
\end{remark}

\begin{corollary}\label{cor:topinvD} 
Given a nonzero two-dimensional coherent sheaf $\calF$ with compact
support contained in $X$, there exists a unique pair $(r, \beta)\in
\IZ\times H_2(S)$, $r \geq 1$, so that the Poincar\'e duality maps the pair
$(\ch_2(\calF), \ch_3(\calF))$ to $(r\sigma_*([S]), \sigma_*(\beta))$.
\end{corollary} 

\begin{proof} 
Since $\calF$ has compact support contained in $X$, we have 
\[
\sigma_\infty^*(\ch_{2k}(\calF)) = 0 
\]
in $H^{2k}(D_\infty)$ for $2\leq k\leq 4$.  Then the claim follows from
\zcref{prop:topinvC}.
\end{proof}

Since Poincar\'e duality is a canonical isomorphism, by a slight abuse
of notation we will simply write
\[
\ch_2(\calF) = r\sigma_*([S]), \qquad \ch_3(\calF) = \sigma_*(\beta) 
\]
in the following. Similarly, we will implicitly use the Poincar\'e
duality isomorphism $H^2(S) \to H_2(S)$.

We conclude this section with the  following result, which follows from
the Grothendieck--Riemann--Roch theorem for the projection $\rho: Y \to S$.
\begin{proposition}\label{prop:topinvA} 
Let $\calF$ be a nonzero two-dimensional coherent sheaf $\calF$
with compact support contained in $X$ and let $\calE \coloneq
\rho_*\calF$. Then the following relations hold in $H_\bullet(Y,\IQ)$: 
\begin{align}
\ch_2(\calF) & = \rk (\calE) \sigma_*[S]\\
\ch_3(\calF) &  = \sigma_*\left( \ch_1(\calE) + \frac12 \rk(\calE) \, c_1(S) \right)\\
\chi(\calF)
& =\chi(\calE).
\end{align}
\end{proposition}


\subsection{Compactness}\label{sec:compactmoduli} 

In this section we prove a compactness result for moduli spaces of
$\PT_1$-stable pairs on a local fourfold $X$ assuming that $H^0 (\LL_i)=0$
for $1\leq i\leq 2$.


\begin{lemma}\label{lem:noneffective}
Let $D \in {\sf N}_1(S)\setminus\{0\}$  be non-effective divisor class
on $S$.  Then there exists an ample divisor class $H_S \in {\sf N}_1(S)$
so that $H_S \cdot D <0$.
\end{lemma} 

\begin{proof} 
We first show that there exists a real ample divisor class $\eta$
on $S$ so that $\eta\cdot D <0$.  Under the stated conditions, $D$
is a non-zero, non-effective divisor class on $S$.  Since the nef
cone of $S$ is dual to its K\"ahler cone, this further implies that
there exists a nef real divisor class $\eta_0\in {\sf N}_1(S)_\BR$ so
that $\eta \cdot D_i< 0$. Let $A$ be an arbitrary ample real divisor class.
If $A\cdot D \leq 0$, then $\eta \coloneq \eta_0+A$ is ample and $\eta
\cdot D <0$.  If $A\cdot D >0$, then $(\eta_0+\varepsilon A) \cdot D
<0$ for sufficiently small $\varepsilon \in \BR$, $\varepsilon >0$.
Then set $\eta \coloneq \eta_0+\varepsilon A$.

Since $\IQ$ is dense in $\BR$, and the real nef cone of $S$ is finitely
generated, it follows that we can find a rational ample divisor class
$\eta'$ satisfying $\eta'\cdot D<0$. Moreover, there exists a positive
integer $N$, so that $N \eta' \in {\sf N}_1(S)$.
\end{proof} 

For the following lemma we change base via a morphism ${\rm Spec}\,K \to
{\rm Spec}\, \IC$, with $K$ an arbitrary $\IC$-field. For convenience,
this base change will not be explicitly reflected in the notation.
\begin{lemma}\label{lem:compA} 
Assume $H^0 (\LL_i)=0$ for $i=1,2$, and let $Z\subset Y$ be a nonempty
reduced irreducible closed subscheme of $Y$ contained in $X$,
such that the structure sheaf $\calO_Z$ is purely two-dimensional.
Then $Z$ coincides with the scheme-theoretic image
of the zero section $\sigma: S\to Y$ as a closed subscheme of $Y$.
\end{lemma} 

\begin{proof} 
The image of the zero section is the common zero locus of the
tautological sections  $z_i \in H^0(\calO_Y(1) \otimes \rho^*\LL_i)$,
$1 \leq i\leq 2$. Under the stated conditions, it suffices to prove that
the multiplication map
\[
\calO_Z \xrightarrow{z_i\otimes {\rm id}} \calO_Z \otimes \calO_Y(1) \otimes \rho^*\LL_i 
\]
is zero for $1\leq i \leq 2$.

Since $Z$ is reduced irreducible, both the domain and the target of the
above morphism are $H$-stable for any ample divisor class $H$ on $Y$.
Therefore it suffices to show that there exists an ample divisor class
$H$ on $Y$ so that the reduced Hilbert polynomials satisfy the inequality
\[
p_H( \calO_Z \otimes \calO_Y(1) \otimes \rho^*\LL_i ; t) \prec
p_H(\calO_Z; t) 
\]
for $1\leq i \leq 2$. Since the support of $\calO_Z$ is contained in $X$, 
by \zcref{cor:topinvD}, we have 
\[
\ch_2(\calO_Z) = r\sigma_{*}[S], 
\]
for some $r\geq 1$. Let also $\ch_1(\LL_i) = \gamma_i \in H_2(S, \IQ)$.
Then a straightforward computation shows that
\[
p_H ( \calO_Z \otimes \calO_Y(1) \otimes \rho^*\LL_i ) - p_H(\calO_Z) =
\frac{2}{r} \frac{H \cdot \sigma_*(\gamma_i)} {H^2 \cdot \sigma_*([S])} t + \delta,
\]
where $\delta$ consists of degree-zero terms.

In conclusion, it suffices to show that there exists an ample divisor $H$
on $Y$ so that
\[
H \cdot \sigma_*(\gamma_i) <0
\]
for $1\leq i \leq 2$.  Given an ample divisor $H_S$ on $S$, any divisor
of the form $H= \rho^{-1}(H_S) + m [D_\infty]$ is ample for sufficiently
large $m>0$ since $D_\infty$ is relatively ample over $S$.  Thus, since
$D_\infty \cdot \sigma_*(\gamma_i)=0$, it suffices to find an ample
divisor $H_S$ on $S$ so that
\[
H_S \cdot \gamma_i<0
\]
for $1\leq i\leq 2$.
Since $H^0(\LL_i)=0$ by assumption, it follows
that $\LL_i \simeq \calO_S(D_i)$ for a non-zero, non-effective divisor
$D_i$ on $S$.  Then the claim follows from \zcref{lem:noneffective}.
\end{proof} 

Under the same assumptions as in \zcref{lem:compA}, we further obtain:
\begin{corollary}\label{cor:compB} 
Let $\calF$ be a nonzero purely two-dimensional coherent sheaf on $Y$
with compact support contained in $X$.  Then there exist two positive
integers $k_i>0$, $1\leq i\leq 2$, so that the scheme theoretic support of
$\calF$ is contained in the common zero locus of the sections $z_i^{k_i}
\in H^0 (\calO_Y(1) \otimes \rho^*\LL_i)$, $1\leq i \leq 2$. 
\end{corollary}

\begin{proof} 
Let $Z$ be the scheme theoretic support of $\calF$. Since $\calF$ purely
two dimensional, so is $\calO_Z$. Then \zcref{lem:compA} shows that each
irreducible component of $Z_{\sf red}$ coincides with the image of the
zero section. Hence $\calF$ is set theoretically supported on the zero
section. Since $\calF$ is coherent, this implies the claim. 
\end{proof}

In conclusion, resetting the ground field to $\IC$, we have:
\begin{proposition} \label{prop:compC}
If $H^0 (\LL_i)=0$ for $i=1,2$, then the moduli space $SP_1(X; v)$
of $\PT_1$-stable pairs is proper over $\IC$ for any topological
invariants $v$.
\end{proposition} 

\begin{proof} 
We have to check the valuative criterion for properness. Let $T =
{\rm Spec}(R)$, where $R$ is a d.v.r.\ over $\IC$. Let $o \in T$ be the
unique closed point and let $U =T \setminus \{0\}$. Let the $(\scrF_U,
s_U)$ be a ${\PT}_1$ stable pair associated to an arbitrary morphism
$U \to SP_1(X;v)$.  Since $SP_1(X,v) \subset SP_1(Y; v)$ is open, and
$SP_1(Y)$ is proper over $\IC$, the pair $(\scrF_U, s_U)$ admits a unique
extension $(\scrF_T, s_T)$ to a $T$-flat family of $\PT_1$-stable pairs
on $Y$. Hence it suffices to prove that the restriction $(\scrF_o, s_o)$
to the closed point belongs to $SP_1(X_o; v)$.

For any pair $(k_1, k_2) \in \IZ^2$, $k_1, k_2>0$, let $Z^{k_1,
k_2}\subset Y$ denote common zero locus of $z_i^{k_i}$, $1\leq i \leq 2$.
By \zcref{cor:compB}, the scheme-theoretic support of the sheaf $\scrF_U$
is a closed subscheme of $Z^{k_1,k_2}_U$ for some positive integers $(k_1,
k_2)$.  By flatness, this implies that the scheme-theoretic support of
$\scrF_o$ is a closed subscheme of $Z^{k_1, k_2}_o \subset X_o$. See
for example \cite[Corollary V.2.16]{Coha_Yangians}. This proves the claim.
\end{proof}

\begin{remark} \label{rem:compcond}
If $\LL_1 \otimes \LL_2 \simeq \omega_S$, the vanishing conditions
$H^0(\LL_i)=0$, $1\leq i \leq 2$, are equivalent to
$H^k(\LL_1) =0$ for $k=0,2$, as well as $H^k (\LL_2)=0$ for $k=0,2$,
by Serre duality.
\end{remark}

\section{Gauge theory}\label{sec:4d}

This section consists of a detailed treatment of Coulomb branch
localization in topologically twisted ADHM gauge theory on a toric
surface $S$. Note that $X$ is also toric, hence it admits a natural
rank-four torus action $T_X \times X \to X$.  In the following, we will
work equivariantly with respect to the rank-three subtorus $T\subset T_X$
that preserves the Calabi--Yau structure.

Within this section, $\chi(S, \ )$ denotes the $K$-theoretic
equivariant Euler characteristic, namely the alternating sum of the
characters.  
Our present construction of the Coulomb branch partition function
is based on the assumption that $\chi(S,\LL_1)$ has no constant terms,
which translates into the absence of certain zero modes
in the perturbative part.  (We leave the study of these for future work.)

\subsection{Fields and equations}

Consider a supersymmetric gauge theory on $S$, with gauge group
$\grp{U}(k)$.  Let $V_k$ be the $k$-dimensional vector bundle over $S$,
associated to the principal $\grp{U}(k)$ bundle whose curvature we denote
by $F$.  Let $M$ and $N$ be two copies of the constant $n$-dimensional
vector space $\BC^n$.  The matter content includes two complex adjoint
scalars $B_{i} \in H^0 (S, \End(V_k) \otimes \LL_i)$, for $i=1,2$,
two fundamental scalars $I \in H^0 (S, \Hom(V_k,N))$ and $J \in
H^0 (S, \Hom(N,V_k) \otimes K_S)$, and fermionic scalars $\Upsilon
\in \Pi H^0 (S, \Hom(V_k,M))$ and $\Psi \in \Pi H^0 (S, \Hom(M,V_k)
\otimes K_S)$.  Note that the contribution of fermionic scalars to the
beta function cancels that of fundamental scalars, hence the theory is
conformal.\footnote {In the topological theory, we are allowed to break
spin-statistics.  Such matter arises e.g.\ in the study of supergroup
gauge theory \cite{Dijkgraaf:2016lym}.}

The partition function of the topological theory localizes to $Q$-fixed
point configurations, which in the present case include \cite[section
8]{Nekrasov:2023nai}
\begin{equation}\label{eq:Qfixed}
F^{(2,0)} + [B_1,B_2] + IJ = 0.
\end{equation}
Let $\mred_k$ be (some compactification of) the moduli space of solutions
to our equations modulo gauge equivalence, which is a disjoint union
over topological sectors, labeled by $\ch(F)=(k, c_1(F), \ch_2(F))$.

We want to compute from first principles generating sums
of K-theoretic virtual invariants, schematically of the form
\begin{equation}
 Z_k^C \coloneq \chi (\mred_k,
 \widehat{\cO^{vir} \otimes \wedge^\bullet_M \calF}),
\end{equation}
where $\calF$ denotes the universal object on $\mred_k \times S$.

Our theory makes sense for any K\"ahler surface, but for the rest of
this section we assume that $S$ is toric, and compute $Z_k^C$ by virtual
localization.  Although one could use as observable any equivariantly
closed form combined with polynomials in the curvature (e.g.\ surface
observables, Chern-Simons terms, etc.), we weigh the solutions to the
above monopole equations by
\begin{equation} \label{eq:fug-pt}
I \coloneq \int_S \ch(F) \td(S) (\tau_2 + \omega_S),
\end{equation}
where $\tau_2$ is an arbitrary complex fugacity
and $\omega_S$ is the K\"ahler form on $S$.

\begin{remark}
Up to a sign, the first term is the classical part of Nekrasov--Okounkov,
with a small improvement \cite{Nakajima:2005fg}.  Up to a constant,
it also equals the virtual dimension of the moduli space, as explained
in \zcref{sec:vdim}.
\end{remark}

\begin{remark}
Note that, in the present case, a specific compactification of
the moduli space of solutions to the above monopole equations is
provided by the moduli space of framed ADHM sheaves constructed in
\zcref{cor:pairsADHM}. Then, the above expression for $I$ is related
to the topological invariants of an ADHM sheaf $(\calE, \Phi_1, \Phi_2,
\phi)$ by
\[
I = \chi(\calE) \tau_2 + \int_S \left(\ch_1(\calE) + \frac{r}{2} c_1(S) \right)\omega_S.
\]
Below it will be convenient to view all Chern classes as homology classes 
by Poincar\'e duality. 

We next show that $I$ can be written explicitly in terms of topological
invariants of the associated $\PT_1$-stable pair $(\calF, s)$ on $X$
using \zcref{prop:topinvA}. As shown in loc.~cit., we have
\[
\ch_3(\calF) = \sigma_*(\ch_1(\calE) + \frac{r}{2} c_1(S)),
\qquad \chi(\calF) = \chi(\calE).
\]
Since the K\"ahler form $\omega_S$ is the restriction of the K\"ahler
form $\omega_X$ via the zero section, we obtain
\[
I  =\chi(\calF) \tau_2+ \langle \omega_X,  \ch_3(\calF) \rangle 
\]
where the pairing $\langle\ , \ \rangle$ is defined in \zcref{rem:topinvF}. 

Choosing a basis $\{\gamma^i\}$ in $H^2(S,\IQ)$ with $1\leq i \leq b_2(S)$, 
we define the gauge theory counting parameters as 
\[
\qq \coloneq \exp \tau_2, \qquad 
\QQ_i \coloneq \exp t_i
\]
where $t_i$ are the coefficients of $\omega_S$ with respect to the
basis $\{\gamma^i\}$.  By \zcref{cor:topinvD}, we get $\ch_3(\calF) =
\sigma_*(\beta)$ for a unique element $\beta\in H_2(S)$. Let $\beta_i$
be the coefficients of $\beta$ with respect to the Poincar\'e dual
basis $\{\gamma_i\}$.  Then we obtain
\be \label{eq:gaugeweight}
e^{I} = \qq^{\chi(\calF)} \QQ^\beta, 
\ee
where 
\[
 \QQ^\beta \coloneq
 \prod_{i=1}^{b_2(S)} \QQ_i^{\beta_i}.
\]
This shows that the natural counting parameters in gauge
theory coincide with those used in $\PT_1$-stable pair theory
\cite{Counting_surfaces_I,Counting_surfaces_II}.
\end{remark}

\subsection{Toric data} \label{sec:toricdata}

Here we collect basic facts on toric geometry needed in the
construction of Coulomb branch partition functions for toric
surfaces. We will review the construction of a toric variety as a
symplectic K\"ahler quotient of the linear space $\IC^N$ by the real
torus $U(1)^M$. A good reference is the book by \textcite{audin}.
Abusing the language, by a Hamiltonian action of an algebraic torus
action on a symplectic manifold we will mean a Hamiltonian action of
its maximal compact subgroup.

The rank $N$ torus  $(\BC^N)^\times$ acts naturally on $\IC^N$ by 
\[
x_a \mapsto \q_a x_a
\]
where $(x_1,\dots,x_N)$ are the standard linear coordinates on $\BC^N$.
This is a Hamiltonian action with Hamiltonian function 
$H=\sum_a p_a \ep_a$, where the maximal compact subgroup $U(1)^N \subset
(\IC^{\times})^N$ is parameterized by $\q_a = \exp \beta \ep_a$, and the
momenta $p_a$ are given by $p_a = |x_a|^2$, $1\leq a \leq N$.

Given an integer-valued matrix $Q: \BZ^N \to \BZ^M$, with $0<M <N$,
we construct a second torus action
$(\IC^\times)^M \times \IC^N \to \IC^N$ given by 
\[
 (z_\alpha,x_a) \mapsto (\prod_{\alpha=1}^{M} z_\alpha^{Q_a^\alpha} x_a)
\]
Below we will use the notation $T_Q \coloneq (\IC^{\times})^M$ and we
will parameterize the elements of its real compact subgroup by
\[
 z_\alpha = \exp(\gamma_\al).
\]
Then note that the $T_Q$-action on $\IC^N$ is again Hamiltonian, with
momentum maps
\begin{equation}
\mu_\al (p) = \sum_a p_a Q^\al_a, \quad \al = 1,\dots,M
\end{equation}
Given a regular value $t \in \BR^M$ of $\mu$, we construct the toric
variety $Z$ as the symplectic quotient $Z \coloneq \mu^{-1}(t)/T_Q$. This
is a smooth K\"ahler variety of complex dimension $d \coloneq \dim(Z)
= N-M$. Moreover, as a complex variety, $Z$ also admits a GIT quotient
presentation $Z \simeq (\IC^N \setminus \Delta)/T_Q$, where $\Delta
\subset \IC^N$ denotes the unstable locus.

Next note that the $(\IC^\times)^N$-action on $\IC^N$ induces
a rank $d$ torus action $T_d \times Z\to Z$, where $T_d \coloneq
(\IC^\times)^{d}$. For each $0\leq p \leq d$, the set of irreducible
subvarieties of $Z$ preserved by the $T_d$-action is finite, and it will
be denoted by $\Delta_p (Z)$. In particular, $\Delta_0(Z)$ coincides with
the $T_d$-fixed locus. Moreover, note that the subvarieties belonging to
$\Delta_{d-1}(Z)$ are precisely the toric divisors $\{x_a=0\} \cap Z$,
$1\leq a \leq N$.

We also recall that the second integral cohomology group $H^2(Z)$ is
isomorphic to $\IZ^M$, and the odd degree cohomology vanishes. Moreover,
the Euler characteristic of $Z$ is given by $\chi(Z)=|\Delta_0(Z)|$;
if $Z$ is a smooth compact surface we also have $\chi(Z) = |\Delta_1(Z)|$.

The Picard group of $Z$ is freely generated by a collection of line
bundles $L_\al$, $1\leq \alpha\leq M$, associated to the canonical
characters of $T_Q$. Moreover, by construction, there is also a canonical
collection of  line bundles $L_a$ associated to the toric divisors
$\{x_a =0\}$, $1\leq a\leq N$.  For each $1\leq a \leq N$, one has a
line bundle isomorphism
\be\label{eq:lbrel}
L_a \simeq \otimes_{\alpha=1}^M L_\alpha^{Q_a^\alpha}, 
\ee
and we  also have $c_1 (Z) = \sum_{a=1}^Nc_1 (L_a)$.  In addition,  note
that each line bundle $L_a$ has a natural $T_d$-equivariant structure.

Every fixed point $v \in \Delta_0 (Z)$ defines a set $I_v \subset
\{1,\dots,N \}$ of cardinality $d$, such that $p_a = 0$ at $v$ for
every $a \in I_v$.  Denote by $L_v: \BZ^N \to \BZ^d \subset \BZ^N$ the
map projecting as the identity on $I_v$, thus labeling its elements,
and $L_v^c = 1-L_v$.

We require that $Q$ is such that fixed points are isolated, i.e.\
$\Delta_0(Z)$ consists of finitely many closed points. Then the
restriction $Q_v \coloneq L_v^c\cdot Q^t$ to the lattice $\{
(n_1,\dots,n_N) \in \BZ^N \mid n_a = 0\,\forall a \in I_v\}$ is
invertible, and it satisfies $Q_v \gamma = L^c \ep$, since $v$ is a
fixed point.  The momenta at $v$ satisfy $p = L^c p$.  Plugging their
value into the moment map, we compute the Hamiltonian $H_v = \ep^t L^c
Q^{-t}_v t$.  We define local weights at $v \in \Delta_0$
\begin{equation}
w: \BZ^N \to \BZ^d, \quad
w = L_v - L_v \cdot Q^t \cdot Q^{-1}_v \cdot L^c_v
\end{equation}
such that they compute the character of the tangent space to $Z$
at $v$.  Finally, let $q_\mu = \exp (\beta \sum_a w_{a,\mu} \ep_a)$
for $\mu=1,\dots,d$, as well as $P \coloneq \prod_{\mu=1}^d (1-q_\mu)$
and $Q \coloneq \prod_{\mu=1}^d q_\mu$.

In the following $S$ will be a smooth compact toric surface
constructed as a toric quotient $$(\IC^N \setminus \Delta)/T_Q~,$$
where $T_Q= (\IC^\times)^{N-2}$. At the same time, $X$ will be a
toric Calabi--Yau fourfold of the form $(\IC^2 \times (\IC^N \setminus
\Delta))/T_{\widetilde Q}$, where the integral matrix ${\widetilde Q}:
\IZ^{N+2}\to \IZ^{N-2}$ has block form
\[
{\widetilde Q} = \left(Q \big| \Lambda\right) 
\]
for a linear map $\Lambda: \IZ^{2} \to \IZ^{N-2}$. The above toric
presentation implies that $X$ is naturally identified with the total space
of a direct sum of two line bundles $\LL_1\oplus \LL_2$ on $S$, where
\[
\LL_i \simeq \prod_{\alpha=1}^M L_\alpha^{\Lambda_i^\alpha}. 
\]
Furthermore, as shown above, one has generic torus actions $T_2 \times S$
and $T_4 \times X\to X$ with finite fixed locus.  Given the above toric
construction, the $T_4$-action preserves the image of the zero section
$S\to X$ and the induced action $T_4 \times S\to S$ is compatible with
the $T_2$-action. More precisely, one has an exact sequence
\be\label{eq:torseq}
1 \to T_V \to T_4 \to T_2 \to 1
\ee
of algebraic tori, with $T_V \simeq \IC^\times \times \IC^\times$, and
the induced $T_4$ action on $S$ factors through the projection $T_4 \to
T_2$. In particular, the stabilizer of any closed point in $S$ coincides
with $T_V$.  Moreover, by construction, each line bundle $\LL_i$ has
a natural $T_4$-equivariant structure for $1\leq i\leq 2$.  Finally,
we will denote by $T \subset T_4$ the stabilizer of the Calabi--Yau
structure on $X$, which will be referred to as the Calabi--Yau torus.

\subsection{Coulomb branch and Fixed points} \label{sec:coulombfixed}

This section outlines our strategy for the construction of the Coulomb
branch partition function of topologically twisted ADHM theory on a
toric surface.

Physically, our intuition suggests that we can study our theory both in
the Coulomb branch and Higgs branch, and the computations should agree.
On the former, we first treat each vacua as isolated, and apply cluster
decomposition principle.  For a non-compact space, this is the whole
story.  For a compact space, we are
moreover instructed to integrate over the Coulomb branch, since there
is quantum mechanical tunneling between different vacua.

More concretely, we will assume that the partition function of
the topological $\grp{U}(k)$ gauge theory on a compact surface $S$
reduces to an integral over a moduli space $\mcou_k$, which at present
does not admit a rigorous mathematical construction.  Heuristically,
this moduli space closely resembles a master space \cite{Thaddeus94}.
By cluster decomposition, we expect
the partition function to be obtained by gluing local data defined in
the canonical affine charts on $S$ centered at the $T_2$ fixed points
$v\in \Delta_0(S)$.

The main working hypothesis is that the local data associated to any
vertex $v \in \Delta_0(S)$ will be constructed by equivariant integration
on moduli spaces of framed torsion-free sheaves on $\BP^1 \times \BP^1$
containing the toric chart centered at $v$ as the complement of the
normal crossing divisor
\[
\Delta_\infty:= \BP^1 \times \{ \infty\} \cup \{\infty\}\times \BP^1.
\]
With suitable framing conditions along $\Delta_\infty$, these moduli
spaces admit a $T_2 \times T_f$ torus action, where the $T_2$ action
is induced from the action on $S$, while  $T_f \coloneq (\IC^\times)^k$
is induced by the canonical action on the framing.

Moreover, the $T_2 \times T_f$-fixed points are classified by $k$-uples of
asymptotic 2d partitions. For completeness, an asymptotic $2d$ partition
$\tl$ is defined as a subset $\tl \subset \IZ_{\geq 0}\times \IZ_{\geq
0}$ satisfying the conditions:
\begin{itemize}
\item If $(i,j)\in \tl$ and $i\geq 1$, then $(i-1,j)\in \tl$.
\item If $(i,j)\in \tl$ and $j\geq 1$  then $(i,j-1)\in \tl$.
\item For sufficiently large $i$, the set  $\tl\cap (\{i\}\times \IZ)$ is finite
and has constant size $\ell_1\geq 0$,  independent of $i$.
\item For sufficiently large $j$, the set $\tl\cap (\IZ\times \{j\})$ is finite
and has constant size $\ell_2\geq 0$, independent of $j$.
\end{itemize} 
As opposed to finite 2d partitions, we do not require the total number
of points in $\tl$ to be finite. Geometrically, asymptotic 2d partitions
classify $T_2$-fixed points in the Hilbert scheme of ideal sheaves $\calI
\subset \calO_{\BP^1\times \BP^1}(-\ell_1,-\ell_2)$ so that the quotient
$\calO_{\BP^1\times \BP^1}(-\ell_1,-\ell_2)/\calI$ is zero-dimensional.

For ADHM gauge theory, the $T_2 \times T_f$-action is lifted to an action of
$T_4\times T_f$, where the action of $T_4$ is induced via the construction
of the gauge theory as an effective theory of $D$-branes supported on
the zero section $S\to X$. In particular, the subgroup $T_V \subset
T_4$ introduced in \zcref{eq:torseq} acts trivially on the instanton
moduli space, but nontrivially on the adjoint fields $B_1$,
$B_2$ associated to deformations in the normal directions.

Therefore, for each vertex $v\in \Delta_0(S)$, the local partition
function will consist of a  perturbative part, and a $K$-theoretic
instanton sum defined in terms of asymptotic 2d partitions.  In addition,
one has to provide a gluing algorithm of the local data, along edges $e
\in \Delta_1(S)$.  The detailed construction is explained below.

\subsection{Gluing algorithm}

In order to fix notation, for each vertex  $v \in \Delta_0(S)$ note that
the induced $T_4$-action on the cotangent space $T_v^*(X)$ is diagonal
with respect to the natural toric frame.  We will denote by $q_{v,i}$,
$1\leq i\leq 4$ the associated characters, choosing notation so that
$q_{v,1}$, $q_{v,2}$ will be associated to the action on $T_v^*(S)$,
while $q_{v,3}$, $q_{v,4}$ are associated to the action along the conormal
directions.  Hence $q_{v,3}$ and $q_{v,4}$ also encode the $T_4$ action
on the adjoint fields $B_1$, $B_2$ respectively.
For simplicity we will
write $q_{v,i} = q_i$, working in an arbitrary fixed local chart.
Note that the Calabi--Yau torus $T\subset T_4$ is defined by
\[
q_1q_2q_3q_4 =1. 
\]

For any $i=1,\dots,k$, let $\tK_i$ be a collection $\tl_i=(\tl_{i,v})_{v
\in \Delta_0(S)}$, where, for each vertex, $\tl_{i,v}$ is an asymptotic (2d)
partition.  To any partition $\tl$ we associate its character
\[
\ch_{\tl} (q_1, q_2) \coloneq \sum_{(a,b)\in \tl} q_1^{a-1} q_2^{b-1},
\]
which can have poles at $q_\mu \to 1$, since the partition is allowed to
be of infinite size.  Then we define equivariant fluxes
\[
\tn_{\mu,i} (\tl_i) \coloneq \lim_{q_\mu \to 1} (1-q_\mu) \ch_{\tl_{i,v}}(q_1,q_2) \in \BZ_{\geq 0}
\quad \mu=1,2.
\]
(The integer is also naturally the character of a one-dimensional
partition.)  Therefore, for any collection $\tK_i$ we obtain two
$\chi$-tuples of equivariant fluxes $\tn_{\mu,i}$, $\mu=1,2$, where
$\chi \coloneq \chi(S)$.  Gluing compatibility of $\tK_i$ means that
actually $\tn_{\mu,i}$ descend from a globally defined $\tn_{a,i}$,
for $a=1,\dots,|\Delta_{1}(S)|$.  We use them to define line bundles
$L_i = \otimes_{e \in \Delta_1} L_e^{\tn_{e,i}}$, for $i=1,\dots,k$.
Topologically, we have $c_1 (L_i) = \sum_{a,\al} \tn_{a,i} Q^\al_a c_1
(L_\al)$.  Let
\begin{equation}
\xi_{\al,i} \coloneq \int_{\BP^1_\al} c_1 (L_i) = \sum_a Q_\al^a \tn_{a,i}
\qquad \al=1,\dots,M
\quad i=1,\dots,k
\end{equation}
as well as
\begin{equation}
\ch_2(L_i) = \frac12 \sum_{e,e' \in \Delta_1} \tn_{i,e} \tn_{i,e'} D_e \cdot D_{e'}
\end{equation}
such that the sum $L = \oplus_{i=1}^k L_i$ has character $\ch (L) =
\sum_{i=1}^k \ch L_i$.  In other words, topologically $c_1 (L_i) =
\sum_\al \xi_{\al,i} c_1 (L_\al)$ and $\ch(L_i) = \exp c_1(L_i)$.

\begin{lemma} \label{lem:change}
The $T_d$-equivariant character 
\begin{equation} \label{eq:qtn}
q^{\tn} \coloneq \ch L_i = \exp \beta (w \ep)^t (L_v\tn)
\end{equation}
splits as the sum of a globally defined part
that depends on the \emph{equivariant} fluxes
and a local (i.e., vertex-dependent)
part that only depends on the topological values,
\begin{equation}
\beta^{-1} \log q^{\tn} = X - X_v
\end{equation}
where $X = \sum_{e\in\Delta_1} \ep_e \tn_{e}$ is related to the
Hamiltonian by replacing momenta with equivariant fluxes and $X_v =
\ep^t L_v^c Q^{-t}_v (Q\tn)$ is similarly related to $H_v$, by the
replacement $t \to Q\tn$.
\end{lemma}
\begin{proof}
We work locally at $v \in \Delta_0$.  Let us first check that our
definition makes sense.  Since $T_d$ is the quotient of the big
torus by $T_Q$, we have $\ch(L_i) = \ep_a \tn_a - \gamma_\al \xi_\al$.
By recalling $\xi = Q\tn$ and $Q_v \gamma = L^c \ep$, we immediately get
$\tn^t (\ep-Q^t \gamma)= \tn^t (1-Q^t Q_v^{-1}L^c) \ep= \tn^t L(1-Q^t
Q_v^{-1}L^c) \ep$, and we thus recover $w$.

From the definition, we compute
\begin{equation}
\text{lhs} = \ep^t L_v\tn - \ep^t L_v^c Q^{-t}_v QL_v\tn
\end{equation}
Using the relations $L_v+L^c_v=1$ and $L_vL_v^c=0$, we get 
\begin{equation}
\text{lhs} = \ep^t \tn - \ep^t L_v^c \tn - \ep^t L_v^c Q^{-t}_v Q(1-L_v^c)\tn
\end{equation}
We claim that $L_v^c = L_v^c Q^{-t}_v Q L_v^c$
follows from the definition of $Q_v$.
This implies that
\begin{equation}
\text{lhs} = \ep^t \tn - \ep^t L_v^c Q^{-t}_v (Q\tn)
\end{equation}
where we recover the definition of $X$ and $X_v$.
\end{proof}

The renormalized $\tK$ then give $\chi\cdot k$-tuples of (finite size)
Young diagrams $\Lambda = (\vec{\lambda}_1, \dots, \vec{\lambda}_\chi)$,
with $\vec{\lambda}_v = (\lambda_{1,v}, \dots, \lambda_{k,v})$, as
follows: at a fixed point $v \in \Delta_0 (S)$, identifying a partition
$\lambda$ with its character
\begin{equation}
K_i = \ch \lambda_i
= \sum_{(a,b) \in \lambda_i} q_1^{a-1} q_2^{b-1}
\end{equation}
we introduce regularized partitions
\begin{equation} \label{eq:regK}
K_i = q^{-\tn_i} \left( \tK_i - \frac{1-q^{\tn_i}} {P} \right)
\end{equation}
One can check that $\lambda$ so defined is a finite-size partition
(Young diagram).
\begin{figure}[htb]
\centering
\begin{tikzpicture}
    \node[cell, minimum height=2cm, minimum width=6cm] at (0,0) {$n_1$};
    \node[cell, minimum height=6cm] at (0,0) {$n_2$};
    
    \node[cell] at (1,4) {$q_2^2$};
    \node[cell] at (1,3) {$q_2$};
    \node[cell] at (1,2) {$1$};  
    \node[cell] at (2,2) {$q_1$};
    \node[cell] at (2,3) {$q_1 q_2$};
    \node[cell] at (3,2) {$q_1^2$};
    \node[cell] at (4,2) {$q_1^3$};
\end{tikzpicture}
\caption{An example of renormalized $\tK$.
We have $\tn_1 = 1 + q_2$, $\tn_2=1$ and
$\lambda = 1 + q_1 + q_1^2 + q_1^3 + q_2 + q_1 q_2 + q_2^2$.
One can check that \zcref{eq:regK} is satisfied, since
$\tK = \frac{1+q_2}{1-q_1} + \frac{q_2^2}{1-q_2} + \lambda q_1 q_2^2$.
The case $d=2$ is special, as the renormalized partition
is still a partition (this is not true for $d>2$).}
\end{figure}
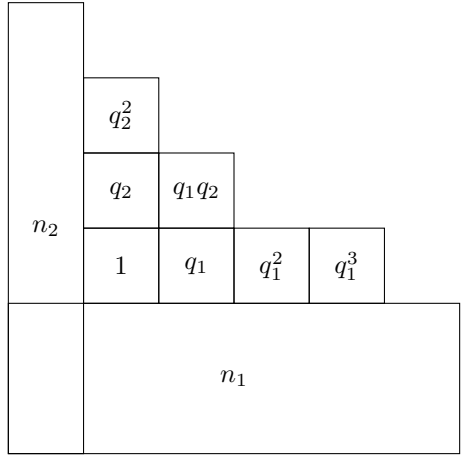

We conclude that tuples $\tK_i$ are in one-to-one correspondence with
pairs $(\lambda_i,\tn_i)$ of Young diagrams and equivariant fluxes.
The fixed point set $\mcou_k^T$ is thus a disjoint union over topological
sectors
\begin{equation}
\mcou^T_{k,ch_1,ch_2} = \left\{ (\tn,\Lambda) \mid
ch_{1,\al} = \sum_i \xi_{i,\al} \, \forall \al, \,
ch_2 = \ch_2(L) - \sum_{v} |\vec\lambda_{v}| \right\}
\end{equation}
\begin{remark}
Note that each of these sets has a finite number of elements (fixed
points), due to the fact that our starting point being $\tK$ implies that
$\tn$ are non-negative integers.  (Of course the differences $\tn_i -
\tn_j$ can be negative.)  Hence no regularization of the sum over $\tn$
is required.
\end{remark}

Using the relations
\begin{equation}
c_1 (L) c_1(S) = \sum_{i=1}^k \sum_{e,e' \in \Delta_1}
\tn_{e,i} D_e \cdot D_{e'}
\end{equation}
as well as
\begin{equation}
c_1 (L).\omega = \sum_{i=1}^k \sum_{e \in \Delta_1}
\tn_{e,i} \vol(D_e)
\end{equation}
we can evaluate the expression \zcref{eq:fug-pt} at $p\in\mcou_k^T$: we get
\begin{equation}
I_p = f_k (\xi) - \tau_2 \sum_{v=1}^\chi |\vec \lambda_v|
\end{equation}
where
\begin{equation} 
f_k (\xi) = \ch(L) \td(S) (\tau_2 + \omega)
\end{equation}
depends on $\tn_i$ only via $\xi_{i,\al}$.  We let $z_{cl} = \exp f_k$
(classical parts), so that $\exp I_p = \qq^{-|\Lambda|} \cdot z_{cl}$.

\subsection{Characters}

At a fixed point $p$, the pullback of the universal sheaf
$\cE$ via the inclusion decomposes as
\begin{equation}
\cE_v = \oplus_{i=1}^k \tilde \cE_i,
\end{equation}
and it has character
\begin{equation}
\ch \tilde \cE_i = \tb_i(1 - P \tK_i)
\end{equation}
where $\tb_i$ are (globally defined) Coulomb branch parameters,
i.e.\ characters of some framing space.
Plugging in the new $K$ from \zcref{eq:regK},
we see that
\begin{equation}
\tb_i (1 - P \tK_i) = q^{\tn_i} \tb_i (1 - P K_i)
\end{equation}
and if we define the shifted Coulomb parameters as
$b_i \coloneq \tb_i q^{\tn_i} = \tb_i \ch L_i$,
as well as $W = \sum_{i=1}^k b_i$ and $V = \sum_i b_i \ch \lambda_{v,i}$,
then we get
\begin{equation} \label{eq:univ}
\ch \cE_v = W - P V
\end{equation}
Let $b_{ij} \coloneq b_i / b_j$ for future reference.

\begin{remark}
Let $\calO = \otimes_{i=1}^k L_i$.
Replacing $\cE$ by $\cE' = \cE \otimes \calO^{-1/k}$,
has the effect of shifting Coulomb parameters in such a way
that adjoint products $b_{ij} = b_{ij}'$ do not change,
while fundamental ones do, and $\prod_i b'_i = \prod_i \tb_i$.
\end{remark}

The character of the virtual tangent space to $\mcou_k$ at $p$ is
computed by the Grothendieck--Riemann--Roch--Atiyah--Singer index theorem
(see e.g.\ the review by \textcite{Pestun:2016qko}). Taking into account the
contribution from the gauge field, adjoint matter and fundamental matter,
respectively, we get
\begin{equation} \label{eq:virtchar}
T_p =  - \chi(S, \cE\otimes\cE^*)
+ \chi(S, \cE \otimes \cE^* \otimes \LL_1)
+ \chi(S, \cE \otimes (N-M)^*)
\end{equation}
where $(\ )^*$ denotes the dual of $( \ )$, and $\chi(S,\ )$ is the
$K$-theoretic equivariant Euler characteristic.  A vector multiplet
contributes
\begin{equation} \label{eq:indvm}
- \chi(S, \cE\otimes\cE^*) =
- \int_S \ch(\cE \otimes \cE^*) \td_S
= \sum_{v \in \Delta_0} \left( T_{adj} - \frac{WW^*}{P^*} \right)
\end{equation}
where the fixed point virtual character
\begin{equation}
T_{adj} = W V^* + Q W^* V - P VV^*
\end{equation}
is computed via \zcref{eq:univ} and $\td_S = 1/P^*$ localized at $v$.
(For flat space, we follow standard notation \cite{Nekrasov:2015wsu}.)
\begin{remark}
For a $\grp{U}(1)$ theory, the adjoint part is independent of $b$,
hence of the flux.
\end{remark} 
Adjoint matter gives
\begin{equation} \label{eq:indad}
\int_S \ch(\cE \otimes \cE^* \otimes \LL_1) \td_S
= - \sum_{v \in \Delta_0} q_3 \left(T_{adj} - \frac{WW^*}{P^*} \right)
\end{equation}
\begin{remark}
In order to address the apparent disparity between $q_3$ and $q_4$
in the above expression, recall that our matter content is determined
via D-brane physics by the geometry of the Calabi--Yau fourfold $X$,
which is the total space of  the rank-two bundle $\LL_1 \oplus \LL_2 \to
S$. In particular, the adjoint fields $B_1$, $B_2$ encode deformations
in the normal directions to $S\subset X$ associated to $\LL_1$, $\LL_2$
respectively. Moreover, the local Calabi--Yau condition requires the
torus characters to satisfy $q_1 q_2 q_3 q_4=1$.  Hence contribution from
$(B_1,B_2)$ likewise satisfies $q_3 T_{adj} = (q_4 T_{adj})^*$.
\end{remark}
Fundamental matter contributes\footnote
{For our purposes, we do not need to twist masses by $K_S^{1/2}$,
since the reference value is arbitrary anyway, nor turn on background fluxes.}
\begin{equation} \label{eq:indfm}
\int_S \ch(\cE \otimes (N-M)^*) \td_S
= \sum_{v \in \Delta_0} \left( T_{f} + \frac{W(N-M)^*}{P^*} \right)
\end{equation}
where
\begin{equation}
T_f = V (M-N)^* Q
\end{equation}
and we identify a vector bundle with its character.

\begin{remark}
We are mostly interested in the case $n=1$, for which we can set the
mass $N=1$.  This is because only $\grp{PGL}(1)$ acts on the Hilbert
scheme of points of $\BC^2$.
\end{remark}

Overall the field content gives
\begin{equation}
T = (1-q_3) T_{adj} + T_f
\end{equation}
and the perturbative contribution is
\begin{equation}
T_{pert} = \frac {W(N-M)^* - W^*W (1-q_3)} {P^*}
\end{equation}
The overall character is thus $T_p = \sum_v (T_{pert} + T)$,
and it is a Laurent polynomial (since $S$ is compact), namely
$T_p \in \BZ[\q_1^{\pm1},\dots,\q_N^{\pm1},\tb_1^{\pm1},\dots,\tb_k^{\pm1},M^{\pm1}]$.

\begin{definition}
Let $\chi (\tn) \coloneq \chi (S,\calO (m) )$, where\footnote
{For a compact surface, $|\Delta_0| = \chi = N$.} $\calO (m) =
\otimes_{a=1}^\chi L_a^{\tn_a}$,
and $\chi_3(m) \coloneq \chi_S(\calO (m) \otimes \LL_1)$.
\end{definition}

We now derive an explicit formula for the perturbative part.
By equivariant localization, one obtains
\begin{equation}
\chi (\tn) =
\sum_{v \in \Delta_0} \frac{q^m} {P^*}~.
\end{equation}
Then the perturbative part reads
\begin{equation} \label{eq:pertpart}
\sum_v T_{pert} =
\sum_i \tb_i \chi(\tn_i) (1-M^*)
- \sum_{i \neq j} \tb_{ij} [\chi(\tn_i-\tn_j) - \chi_3(\tn_i-\tn_j)]
- k + k\chi_3(0)
\end{equation}
where $\chi_3 (0) = \chi(S,\LL_1)$, and we used the fact that
\[
\chi(0) = \chi_S(\calO_S) = 1
\]
since $S$ is toric.

\begin{remark} \label{rem:zeromodes} 
Let
\[
\chi_3 (0) = \chi_3^+(0) - \chi_3^-(0)
\]
where $\chi_3^\pm(0)$ are Laurent polynomials in the equivariant
variables with positive coefficients.  Then note that
if $\chi_3^+(0)=0$, then the partition function admits a
non-equivariant limit.  In particular, this will hold provided that the
line bundle $\LL_1$ satisfies the vanishing condition $H^k(\LL_1)=0$
for $k=0,2$.  By \zcref{prop:compC, rem:compcond}, if this is the case,
then the moduli space of $\PT_1$-stable pairs $SP_1(X,v)$ is compact
for all topological invariants.  Hence the associated partition function
of equivariant $K$-theoretic invariants admits a non-equivariant limit
as well.
\end{remark}

\subsection{Sanity check} \label{sec:vdim}

We want to check that the expected dimension as computed from
$\mcou_k$ matches the virtual dimension from PT theory.
\begin{lemma}
Assuming $\chi_S(\LL_1)=0$, 
the non-equivariant limit of the index $T_p$ in \zcref{eq:virtchar},
neglecting the $M$ dependent terms, gives $\int_S \ch(F) \td(S) - k^2$.
\end{lemma}
\begin{proof}
We work with non-equivariant Euler characteristics, and set $M=0$.

The first two terms almost cancel, except for their perturbative parts,
which give (after splitting terms with $i=j$ and $i \neq j$)
\[
\sum_{i,j} (\chi(\tn_{ij}) - \chi_3 (\tn_{ij}))
= (\chi(0)-\chi_3(0)) (k+2\frac{k(k-1)}{2})
\]
where we used the fact that $\chi_S(M) + \chi_S(M^*) - \chi_S(M
\otimes \LL_1) - \chi_S(M^* \otimes \LL_1) = - \int c_1(\LL_1)
(c_1(\LL_1)+c_1(S)) = 2 (\chi(0)-\chi_3(0))$ for any line bundle $M$.
Since  $\chi_S(\LL_1)=0$ under the current assumptions, we get $k^2
\chi(0) = k^2 \td_S$, and $\int \td_S = 1$ for a toric surface.

Fundamental matter contributes $\int\ch(F)\td(S)$.
\end{proof}

\subsection{Partition function} \label{sec:part_fct}

\begin{definition}
We call $P \in \BZ[\q_1^{\pm1},\dots,\q_N^{\pm1},\tb^{\pm1},M^{\pm1}]$
\emph{movable} if the constant term is absent.  The plethystic map is
defined on movable Laurent polynomials, by sending products to sums.
Write it as a finite sum $P = \sum_i m_i x_i$ where $x_i$ are monomials
with unit coefficient and $m_i$ are integers.  Then
\begin{equation}
\hat a (P) \coloneq \prod_i \br{x_i}^{-m_i},
\quad \br{x} \coloneq x^{1/2} - x^{-1/2}
\end{equation}
In our conventions, it sends positive to denominator.
\end{definition}

After taking the plethystic of everything by the
Atiyah--Bott--Berline--Vergne machinery, we define the perturbative
\begin{equation}
z_{pert} = \hat a ( k + \sum_{v=1}^\chi T_{pert} )
\end{equation}
and instanton parts
\begin{equation}
z_{inst} = \prod_{v=1}^\chi \sum_{\vec \lambda_v} \qq^{- |\vec \lambda_v|} \hat a (T)
\end{equation}
and we get
\begin{equation} \label{eq:integrand}
\sum_{p \in \mcou_k^T} \hat a (T_p+k) \exp I_p
= \sum_{\tn} z_{cl} \cdot z_{pert} \cdot z_{inst}
\end{equation}
\begin{remark}
The inverse of group-like Vandermonde determinant
$\prod_{i \neq j} \hat a (-\tb_{ij})$,
which one would expect on physical grounds,
is actually canceled by some $+2$ zero modes,
so our index, which is the end result, does not see it.\footnote
{We thank Roman Mauch and Lorenzo Ruggeri for discussions on this point.}
\end{remark}

Using \zcref{lem:change} we make the change of variables
\begin{equation} \label{eq:anew}
\tb_i \eu^{\beta X_i} = \anew_i
\end{equation}
in the integrand \zcref{eq:integrand}, so that it depends on $\tn$ only
via $\xi$.  We can thus reduce the computation to contour integrals,
which we are summing over topological (as opposed to equivariant) fluxes,
subject to usual slope stability.  (In terms of $L = \oplus_i L_i$,
this means that the integrals $\sigma_i \coloneq \int_S c_1(L_i) \wedge
\omega = \sum_\al \xi_{i,\al} t_\al$ must satisfy $\sigma_1 -\sigma_2
\geq 0$ for $k=2$ and similar inequalities for higher rank, for $\xi$
to be in the set $\mathrm{st}(t)$.)  This gives
\begin{equation} \label{eq:zk}
Z_k^C = \sum_{\xi \in \BZ^{k \cdot M}_{\geq0}\cap \mathrm{st}(t)}
 \oint_{\calC_\ep (\xi)} d\anew\,
 z_{cl} \cdot z_{pert} \cdot z_{inst}
\end{equation}
with measure $d\anew \coloneq \prod_{i=1}^k \frac{d\anew_i}{\anew_i}$.
For the matter content described above, we take the integration
contour $\calC$ that encloses only poles from the perturbative part.
(The expression for the character $T_{pert}$ is known in closed form,
so it is simple to read off poles from it.)  In conclusion, we are led
to \zcref{conj:main}.

\section{Case study: local P2}\label{sec:localp2}

An interesting example where we can test our claim is the total space
of the sum of two line bundles $\cO(-2) \oplus \cO(-1)$ over $S=\BP^2$,
with Kähler parameter $t>0$.

\subsection{The geometric data}

The charge matrix for $X$ is $Q = (1,1,1,-2,-1)$, while the one for $S$
is obtained by truncating it to the first three columns.

There are three fixed points:
\begin{equation}
\begin{array}{c|c|c|c}
v \in \Delta_0 & \text{coord} & \text{symmetry} & H \\
\hline
v_1 & x_1 \neq 0 & \q_2/\q_1, \q_3/\q_1, \q_4 \q_1^2, \q_5 \q_1 & t \ep_1 \\
v_2 & x_2 \neq 0 & \q_1/\q_2, \q_3/\q_2, \q_4 \q_2^2, \q_5 \q_2 & t\ep_2 \\
v_3 & x_3 \neq 0 & \q_2/\q_3, \q_1/\q_3, \q_4 \q_3^2, \q_5 \q_3 & t\ep_3
\end{array}
\end{equation}
We can always e.g.\ set $\ep_3 = 0$, since the effective torus
acting on $\BP^2$ is two-dimensional.  Label toric divisors of $S$ as
$e_a \colon v_b \to v_c$ for $(abc)$ a permutation of $(123)$.
In this case $\xi_i = \sum_{e\in\Delta_1} \tn_{e,i}$ and $X_i =
\ep_1\tn_{1,i}+\ep_2\tn_{2,i}+\ep_3\tn_{3,i}$.  Then $q^{\tn}$ reads
(dropping the index $i$ for brevity)
\begin{equation}
\begin{array}{c|c}
v \in \Delta_0 & \beta^{-1} \log q^{\tn} \\
\hline
v_1 & (\ep_2-\ep_1)\tn_2 +(\ep_3-\ep_1)\tn_3 = X - \ep_1\xi \\
v_2 & (\ep_1-\ep_2)\tn_1 +(\ep_3-\ep_2)\tn_3 = X - \ep_2\xi \\
v_3 & (\ep_1-\ep_3)\tn_1 +(\ep_2-\ep_3)\tn_2 = X - \ep_3\xi
\end{array}
\end{equation}
\begin{remark}
From this we confirm that the integrand can only depend on $\tn_i$ via $\xi_i$.
For fixed $\xi_i$, we introduce new variables $\anew_i$ as in \zcref{eq:anew},
so that
\begin{equation}
b_i = \anew_i \eu^{-\beta \ep_v \xi_i}
\end{equation}
\end{remark}
For later convenience, let us define holomorphic Euler characteristics
\[
\chi' (\xi) \coloneq \sum_v \frac{\eu^{-\beta \ep_v \xi}}{P^*}
\]
as well as
\[
\chi'_3 (\xi) \coloneq \sum_v \frac{q_3 \eu^{-\beta \ep_v \xi}}{P^*}
\]
so that the perturbative contribution reads
\[
k + \sum_v T_{pert} = \sum_i \anew_i \chi'(\xi_i) (1-M^*)
- \sum_{i \neq j} \anew_{ij} \chi'(\xi_i-\xi_j)
+\sum_{i \neq j} \anew_{ij} \chi'_3(\xi_i-\xi_j) 
\]
using the fact that, for $\BP^2$, $\chi'(0)=1$ and $\chi_3'(0)=0$.

For the $\grp{U}(k)$ gauge theory, the flux and instanton numbers are
\begin{equation}
\xi = \int_{\BP^1} \ch_1(F) = \sum_{i=1}^k \xi_i
\end{equation}
\begin{equation}
\ch_2 = \int_{\BP^2} \ch_2(F) = \frac12\sum_i \xi_i^2
- \sum_v |\vec\lambda_v|
\end{equation}
where the last term is the instanton contribution.  
In order to make a connection with equation \eqref{eq:gaugeweight}, 
note that in the present case one has 
\[
\beta = m [H] 
\]
where 
\begin{equation}
m \coloneq \xi + \frac32 k
\end{equation}
and $[H]\in H_2(\BP^2, \IQ)$ is the hyperplane class. 
Moreover, 
\begin{equation}
n \coloneq \ch_2 +\frac32 \xi + k
\end{equation}
coincides with the Euler characteristic $\chi(\calF)$ in the same equation. 
In conclusion, \zcref{eq:gaugeweight} specializes to 
\[
e^{I} = \qq^n \QQ^m.
\]

Below we verify that the non-equivariant limit of $Z_k^C$ is in agreement
with generating functions $G^{PT_1}_k (\QQ, \qq)$ in \textcite[Section
6.6]{Bae:2024bpx} provided that $y =M$.

\subsection{U1 theory}
As a warm up, consider the $\grp{U}(1)$ theory.
All poles in the perturbative part
come from fundamental matter, since the adjoint part is absent:
\[
1 + \sum_v T_{pert} = \anew_1 \chi'(\xi_1) (1-M^*)
\]
As a heuristic check of our claims, it is easy to verify that the first
term ($\xi=0$, $ch_2=0$) matches equivariantly the face term of DT theory,
with $\lambda=1$.

With the help of some computer program, we can compute \zcref{eq:p2rk1}.

Due to the simplicity of $\grp U(1)$ theory, it may be possible to
get a closed form for $Z_1^C$, as a plethystic exponent, but we did not
investigate this.

\subsection{U2 theory}\label{sec:u2}

For $\grp{U}(2)$ theory, the story is richer.

Let us again run an equivariant check first.
The minimal values are
$(m,n)=(4,4)$ hence $\xi = 1$, $\ch_2 = \frac12$.
DT theory gives
\begin{equation}
- \sum_v \mu^* \frac{1+q_4}{P_{12}}
= - (\q_3 \q_5 \mu^* + \q_2 \q_5 \mu^* +\q_1 \q_5 \mu^* + \mu^*)
\end{equation}
It corresponds to $\tn = ( (1,0,0), (0,0,0) )$ and permutations,
and no instantons.
For $(\xi_1,\xi_2)=(1,0)$, the perturbative part reads
\begin{multline}
2+\sum_v T_{pert} = \anew_1 (\q_1^{-1}+\q_2^{-1}+\q_3^{-1}) + \anew_2
-M^* \anew_1 (\q_1^{-1}+\q_2^{-1}+\q_3^{-1}) - M^* \anew_2 \\
- \anew_{12} (\q_1^{-1}+\q_2^{-1}+\q_3^{-1}) + \anew_{21} \q_5^{-1}
\end{multline}
Taking the residue at $\anew_2=1$
and then $\anew_1=\q_5^{-1}$ we get the result.

Terms with $m=5$ correspond to $\xi=2$:
\begin{equation}
\begin{array}{c|c|c}
n & \ch_2 & (\xi_1, \xi_2, |\Lambda|) \\
\hline
7 & 2 & (2,0,0) \\
6 & 1 & (2,0,1), (1,1,0) \\
5 & 0 & (2,0,2), (1,1,1) \\
4 & -1 & (2,0,3), (1,1,2)
\end{array}
\end{equation}
For $\xi=(2,0)$, we take residues at
$(1,\q_1 \q_5^{-1}), (1,\q_2 \q_5^{-1}), (1,\q_3 \q_5^{-1}), (1,\q_4^{-1})$.

Terms with $m=6$ correspond to $\xi=3$:
\begin{equation}
\begin{array}{c|c|c}
n & \ch_2 & (\xi_1, \xi_2; |\Lambda|) \\
\hline
11 & \frac92 & (3,0;0) \T\\
10 & \frac72 & (3,0;1) \T\\
9  & \frac52 & (3,0;2), (2,1;0) \T\\
8  & \frac32 & (3,0;3), (2,1;1) \T\\
7  & \frac12 & (3,0;4), (2,1;2) \T
\end{array}
\end{equation}
For $\xi=(3,0)$ we take residues at
$(1, \q_1^3\q_2\q_3\q_4)$,
$(1, \q_1^2\q_2^2\q_3\q_4)$,
$(1, \q_1\q_2^3\q_3\q_4)$,
$(1, \q_1^2\q_2\q_3^2\q_4)$,
$(1, \q_1\q_2^2\q_3^2\q_4)$,
$(1, \q_1\q_2\q_3^3\q_4)$,
$(1, \q_3\q_4^{-1})$,
$(1, \q_2\q_4^{-1})$,
$(1, \q_1\q_4^{-1})$,
while for $\xi=(2,1)$ at $(\q_1,\q_1^2\q_2\q_3\q_4)$,
$(\q_2,\q_1\q_2^2\q_3\q_4)$, $(\q_3,\q_1\q_2\q_3^2\q_4)$.

With some more effort, we can compute some terms in the generating sum,
and take the non-equivariant limit: we get \zcref{eq:p2rk2}.  This matches
and extends results in the literature.

\printbibliography

\end{document}